\documentclass[12pt]{amsart}
\usepackage[colorlinks=true, pdfstartview=FitV, linkcolor=blue, citecolor=green, urlcolor=black,filecolor=magenta]{hyperref}

\usepackage{graphicx, verbatim,amsmath,amssymb}

\newcommand{\aaa}{{\mathcal A}}

\newcommand{\fff}{{\mathcal F}}

\newcommand{\eee}{{\mathcal E}}

\newcommand{\ppp}{{\mathcal P}}
\newcommand{\sss}{{\mathcal S}}
\newcommand{\rrr}{{\mathcal R}}

\newcommand{\vvv}{{\mathcal V}}
\newcommand{\ilim}{\varprojlim}

\newcommand{\poly}{\Delta}

\newcommand{\T}{{\bf T}}
\newcommand{\R}{{\mathbb R}}

\newcommand{\N}{{\mathbb N}}

\newcommand{\Z}{{\mathbb Z}}
\newcommand{\C}{{\mathbb C}}

\newcommand{\vecv}{{\vec{v}}}

\newtheorem{thm}{Theorem}[section]
\newtheorem*{thm*}{Theorem}
\newtheorem{cor}[thm]{Corollary}
\newtheorem{lem}[thm]{Lemma}
\newtheorem{prop}[thm]{Proposition}
\newtheorem{dfn}[thm]{Definition}

\everymath{\displaystyle}

\theoremstyle{definition}
\newtheorem{ex}[thm]{Example}

\numberwithin{equation}{section}
 
\begin{document}
\title{Fusion: a general framework for hierarchical tilings of $\R^d$}
\author{Natalie Priebe Frank and Lorenzo Sadun}

\address{Natalie Priebe Frank\\Department of Mathematics\\Vassar
  College\\Poughkeepsie, NY 12604} \email{nafrank@vassar.edu}
\address{Lorenzo Sadun\\Department of Mathematics\\The University of
  Texas at Austin\\ Austin, TX 78712} \email{sadun@math.utexas.edu}
\thanks{The work of the second author is partially supported by NSF
  grants DMS-0701055 and DMS-1101326} 
\subjclass[2010]{Primary: 37B50 Secondary: 52C23,
  37A25, 37B10} 
\keywords{Self-similar, substitution, mixing,
  dynamical spectrum, invariant measures} 
\date{June 29, 2012}

\begin{abstract}
We introduce a formalism for handling general spaces of hierarchical
tilings, a category that includes substitution tilings, 
Bratteli-Vershik systems, S-adic
transformations, and multi-dimensional cut-and-stack
transformations. We explore ergodic, spectral and topological
properties of these spaces. We show that familiar properties of
substitution tilings carry over under appropriate assumptions, and
give counter-examples where these assumptions are not met. For
instance, we exhibit a minimal tiling space that is not uniquely ergodic, with one
ergodic measure having pure point spectrum and another ergodic measure
having mixed spectrum.  We also  exhibit a 2-dimensional tiling space that has pure point
measure-theoretic spectrum but is topologically weakly mixing.
\end{abstract}

\maketitle

\setlength{\baselineskip}{.6cm}

\setcounter{tocdepth}{3}
\makeatletter
\def\l@subsection{\@tocline{2}{0pt}{25pt}{5pc}{}}
\def\l@subsubsection{\@tocline{2}{0pt}{50pt}{5pc}{}}
\makeatother
\tableofcontents

\section{Introduction}
Hierarchical structures are ubiquitous in the real world. Typically
there are a finite number of levels, ranging from the tiny (say,
subatomic particles) to the huge (say, clusters of galaxies). In many
cases the smallest level is so small that it makes sense to extrapolate
mathematically to infinitely small hierarchical structures -- fractals. 
In this paper we consider the complementary situation where the smallest 
scale may not be small, but the largest scale is so large that it makes sense
to extrapolate to infinite size.

There is an extensive literature devoted to expanding hierarchies,
dating back to the 1800s \cite{Prouhet}, with applications to dynamics
dating back to the early 1900s \cite{Morse}.  Most of the aperiodic
sets of tiles that were discovered over the years, from
Berger \cite{Berger} to Robinson \cite{R.Robinson} to
Penrose \cite{Penrose} to Goodman-Strauss \cite{Chaim.annals} and
others, used hierarchy as means of proving aperiodicity. Tiles group
into clusters that group into larger clusters, etc., so that the
resulting patterns exhibit structure at arbitrarily large length
scales and cannot be periodic.

In most of the literature, it is assumed that the hierarchies have
essentially the same structure at each level, so that the system can
be described by a single substitution map. Indeed, there has been tremendous
progress on substitution sequences,
substitution subshifts, and substitution tilings.
However, there is much to be said
about hierarchical systems where the structure is {\em not} necessarily 
repeated at each level.

The idea of studying general hierarchical systems can be seen in the
cut-and-stack formalism of ergodic theory.  The first example of
Chacon \cite{Chacon}, which exhibited a weakly mixing system that was
not strongly mixing, is a fusion of the sort discussed in this paper.
Over the years the technique has been used to construct many
interesting examples, and it has been shown \cite{AOW} that all
interval exchange transformations, and indeed all aperiodic measure
preserving transformations, can be obtained by cutting and stacking.
Cutting and stacking has been generalized to higher dimensions for
$\Z^d$ actions \cite{Rudolph, Aimee-Ayse}, for $\R^d$-actions on
rectangular domains \cite{Cohen}, and for general locally compact
second countable groups \cite{Danilenko-Silva} and amenable groups.
Progress has recently been made on nonstationary Bratteli-Vershik
systems \cite{Fisher, FFT, BKMS}, most of which can be viewed as a discrete
1-dimensional version of the fusion tilings described in this paper
\cite{BJS}.

This paper provides a framework for studying
the ergodic theory and topology of hierarchical tilings. 
Our formalism encompasses, among other things, 
substitution tilings and substitution subshifts,
cut-and-stack transformations, S-adic transformations \cite{Durand},
and stationary and non-stationary Bratteli-Vershik systems
\cite{Fisher, BKMS}.

Taken to extremes, our formalism can be made {\em too} general. 
Without simplifying assumptions, essentially any tiling space
can be viewed as a fusion, and almost any sort of dynamical behavior is 
possible. For instance, Jewett \cite{Jewett} and Krieger \cite{Krieger}
showed that any ergodic measurable automorphism of a non-atomic 
Lebesgue space system can
be realized topologically as a uniquely ergodic map on a Cantor set; in
most cases these can be viewed as subshifts, and hence as fusion 
tiling spaces.  Downarowicz \cite{Down} showed that there exist Cantor
dynamical systems whose invariant measures match an arbitrary Choquet
simplex. 

In this paper 
we identify appropriate hypotheses that preserve the essential
properties of substitutions while applying to more general systems. 
Certain properties, like minimality or unique
ergodicity, hold under very general conditions. Others, like finitely
generated (rational \v Cech) cohomology or pure point spectrum or (on
the other extreme) topological weak mixing, require stronger assumptions. 

In addition, we develop a number of 
examples that show how these properties can be lost 
when the assumptions are too weak. We hope that these examples will help 
to classify fusion tilings, and to better organize our understanding of
tilings in general. 

Some of our proofs are quite simple, yet determining how to apply
the techniques of substitution systems to fusions is far from
trivial.  The key tools for studying substitution systems are
Perron-Frobenius theory and the existence of a self-map that can be
iterated arbitrarily many times.  Neither of these work for general
fusions.
The new methods devised in this paper provide us with more insight
into how properties of tiling {\em spaces} are related to properties
of {\em tilings}.  Some properties of a hierarchical tiling space are
directly related to the geometry of the individual tiles.  Others come
from the details of how the tiles are assembled into bigger and bigger
clusters. Still others can be deduced from coarser numerical data,
such as from the matrices that count how many of each kind of tile
appear in each kind of cluster.  Because the hierarchy in fusion rules
is less rigid than that of their substitutive counterparts,
combinatorics, geometry, algebra, and topology can have effects that
need to be teased apart.  The challenge is to understand which
properties come from which information, and to organize that
information effectively.

\noindent{\bf Acknowledgments.} We thank Mike Boyle, Lewis Bowen, 
Kariane Calta, Amos Nevo, E.~Arthur Robinson, Jr. and Boris Solomyak
for helpful discussions.
The work of L.S.~ 
is partially supported by 
NSF grants DMS-0701055 and DMS-1101326. 

\section{Definitions}

In this work a {\em prototile} is a labelled, closed topological disk
in $\R^d$.  The label, which can be
thought of as a color or a marking,
is necessary when we wish to
distinguish between prototiles that are geometrically similar.  In
general we assume that we have a finite set $\ppp$ of prototiles
to use as building blocks for our tilings.  (This
assumption is useful but not entirely necessary. In a separate work \cite{ILC}
we consider tilings built from an infinite but compact set $\ppp$.)  We also assume that
we have fixed a closed subgroup $G$ of the Euclidean group $E(d)$ that
contains a full rank lattice of translations; this group $G$ will be
used to construct our tiles, patches, and tilings and can also
serve as the group action of our dynamical system.  (The two standard
translation subgroups that appear in tiling theory are $\Z^d$ and
$\R^d$.) It is possible to act on a prototile by an isometry in $G$ by
applying the isometry to the closed set defining the prototile and
carrying the labelling information along unchanged.  A prototile which
has been so moved is called a {\em tile}.  We will abuse notation by
denoting the application of an isometry $g \in G$ to a prototile $p$
as $g(p)$; when the isometry is translation by $\vecv \in \R^d$ we
denote the translated tile by $p + \vecv$.  A {\em $\ppp$-patch} (or 
{\em patch}, for short) of
tiles is a connected, finite union of tiles that only overlap on their
boundaries; the {\em support} of the patch is the closed set in $\R^d$
that it covers.  Two tiles or patches are considered {\em equivalent}
or {\em copies} of one another if there is an element of $G$ taking
one to the other.  A {\em tiling} $\T$ of $\R^d$ is a collection of tiles that
completely cover $\R^d$ and overlap only on their boundaries.

A tiling is said to have {\em finite local complexity (FLC)} with respect to 
the group $G$ if it contains only finitely many connected two-tile patches up 
to motions from $G$.  Most of the literature
on tiling dynamical systems uses finite local complexity as a key
assumption. This work in this paper is limited to FLC fusion tilings. Fusion tilings with  infinite local
complexity (ILC) will be considered in \cite{ILC}.

%

\subsection{Fusion tilings} Given two $\ppp$-patches $P_1$ and $P_2$
and two isometries $g_1$ and $g_2$ in $G$, if the patches $g_1(P_1)$
and $g_2(P_2)$ overlap only on their boundaries, and if the union $g_1(P_1)\cup
g_2(P_2)$ forms a $\ppp$-patch, we call that union the {\em fusion} of
$P_1$ to $P_2$ via $g_1$ and $g_2$.  When we do not wish to specify
the isometries we may call it a fusion of $P_1$ to $P_2$.
Notice that there will be many ways to fuse two patches together and
that we may attempt to fuse any finite number of patches together.  We
may even fuse a patch to copies of itself.  Patch fusion is simply a
version of concatenation for geometric objects.

The idea behind a ``fusion rule'' is an analogy to an atomic model: we
have atoms, and those atoms group themselves into molecules, which
group together into larger and larger structures.  In this analogy we
think of prototiles as atoms and patches as molecules.  Let $\ppp_0 =
\ppp$ be our prototile set, our ``atoms''.  The first set of
``molecules'' they form will be defined as a set of finite
$\ppp$-patches $\ppp_1$, with notation $\ppp_1 = \{P_1(1), P_1(2),
..., P_1(j_1)\}$.  Next we construct the structures made by these
``molecules'': the set $\ppp_2$ will be a set of finite patches that
are fusions of the patches in $\ppp_1$.  That is $\ppp_2 = \{P_2(1),
P_2(2), ..., P_2(j_2)\}$ is a set of patches, each of which is a
fusion of patches from $\ppp_1$.  While the elements of $\ppp_2$
are technically $\ppp$-patches, we can also think of them 
as $\ppp_1$-patches by considering the elements of $\ppp_1$ as
prototiles.  We continue in this fashion, constructing $\ppp_3$ as a
set of patches that are fusions of patches from $\ppp_2$ and in
general constructing $\ppp_n$ as a set of patches which are fusions of
elements of $\ppp_{n-1}$.
The elements of $\ppp_n$ are called  
$n$-fusion supertiles or {$n$-supertiles}, for 
short.\footnote{\label{labeled.supertiles}If we wish, 
we can also add labels to the supertiles, so that the information carried in   
an $n$-supertile is more than just its composition as a patch in a tiling.
This generalization is useful for collaring constructions, as in Section 5.}
We collect
them together into an atlas of patches we call our {\em fusion rule}:

$$\rrr= \left\{\ppp_n, n \in \N \right\} = \left\{ P_n(j) \,\, | \,\,n \in \N 
 \text{ and } 1 \le j \le j_n \right\}.$$    

A patch is {\em admitted by $\rrr$} if a copy of it can be found
inside some supertile $P_n(j)$ for some $n$ and $j$. A tiling $\T$ of
$\R^d$ is said to be a {\em fusion tiling with fusion rule $\rrr$} if
every patch of tiles contained in $\T$ is admitted by $\rrr$.  We
denote by $X_\rrr$ the set of all $\rrr$-fusion tilings.  Given a fusion
rule, we can obtain another fusion rule $\rrr'$ with $j'_n=j_{n+1}$ 
and $P'_n(j) = P_{n+1}(j)$. We simply ignore the lowest level and
treat the 1-fusion supertiles as our basic tiles. The resulting tiling
space is denoted $X^1_\rrr$. Likewise, $X^k_\rrr$ is the space of tilings
obtained from $\rrr$ in which the $k$-fusion supertiles are considered
the smallest building blocks. 

{\bf Standing assumption (for this entire paper):} If
none of the supertiles in $\rrr$ have inner radii approaching infinity
then $X_\rrr$ will be empty, so for that reason we restrict our attention
to fusion rules that have nontrivial tiling spaces.

When $d=1$ and $G=\Z$, with all tiles having unit length,
fusion tilings correspond to Bratteli-Vershik systems, 
modulo complications having to do with edge sequences that have no
predecessors or no successors. See \cite{BJS} for more about the 
correspondence.  (In addition to subshifts, Bratteli-Vershik systems can model  
non-expansive maps on Cantor sets; these can also be viewed as 
1-dimensional fusion tilings, albeit with infinitely many tile types \cite{ILC}.)  

\begin{ex} {\em The Chacon transformation.} In
  \cite{Chacon} there is an early example of a transformation that is
  weakly mixing but not strongly mixing.  The original
  cutting-and-stacking construction is a self-map on an interval; the
  stacking portion can be seen as a sort of fusion.  However for the
  purposes of an immediate example we use the fact that the Chacon
  space can be viewed symbolically using the substitution rule $$ a
  \to a a b a \qquad b \to b, $$ which can be iterated by substituting
  each letter and concatenating the blocks.  If we begin with an $a$
  we have: $$ a \to a a b a \to a a b a \, a a b a \, b \, a a b a \to
  ...$$ In order to make a Chacon tiling of $\R$ we only need to
  assign closed intervals to the symbols $a$ and $b$ and place them on
  the line according to the symbols in a Chacon sequence.

  We can view a Chacon tiling of $\R$ as a fusion tiling as follows.
  Consider $l_a$ and $l_b$ to be two positive numbers and let $a$
  denote a prototile with support $[0,l_a]$ and $b$ denote a prototile
  with support $[0,l_b]$.  (If $l_a = l_b$ then we use the symbols $a$
  and $b$ as labels to tell the tiles apart).  We define $P_1(a) = a
  \cup (a + l_a)\cup (b + 2l_a)\cup (a + 2l_a + l_b)$ and $P_1(b) =
  b$.  The length of $P_1(a)$ is $3l_a + l_b$.  To
  make $P_2(a)$ we simply fuse three copies of $P_1(a)$ and one
  copy of $P_1(b)$ together in the correct order, and of course
  $P_2(b) = b$ still.  The length of the new $a$ supertile is
  three times that of the previous $a$ supertile plus the length of
  $b$.  We continue recursively to construct all of the $n$-fusion
  supertiles.
\end{ex}



\subsubsection{Transition matrices and the subdivision map}
Given a fusion rule $\rrr$ there is a family of {\em transition
  matrices} that keep track of the number and type of $(n-1)$-fusion
supertiles that combine to make the $n$-supertiles.  The transition
matrix for level $n$, denoted $M_{n-1,n}$, has entries $
M_{n-1,n}(k,l) = $ the number of $(n-1)$-supertiles of type $k$, that
is, equivalent to $P_{n-1}(k)$, in the $ n $-supertile of type $l$,
$P_n(l)$.  If there is more than one fusion of $\ppp_{n-1}$-supertiles
that can make $P_n(l)$, we fix a preferred one to be used in this and
all other computations.  For levels $n < N \in \N$, we likewise define
the transition matrix from $n$- to $N$-supertiles as $M_{n,N} =
M_{n,n+1} M_{n+1,n+2} \cdots M_{N-1,N}$.  The $(i,j)$ entry of
$M_{n,N}$ is the number of $n$-supertiles of type $i$ in the
$N$-supertile of type $j$.  Another way to think about this is to
imagine a ``population vector'' $ v \in \Z^{j_N}$ of a patch of
$N$-supertiles: the entries represent the number of $N$-fusion
supertiles of each type appearing in the patch.  Then $M_{N-1,N} 
v$ gives the population of this patch in terms of $(N-1)$-supertiles,
$M_{N-2, N-1} M_{N-1,N} v$ gives the population in terms of
$(N-2)$-supertiles, and $M_{n,N} v$ gives the population of this
patch in terms of $n$-supertiles.

Any self-affine substitution tiling, in any dimension, can be viewed 
as a fusion tiling. An $n$-supertile is what we get by applying the 
substitution $n$ times to an ordinary tile, and can be decomposed into
$(n-1)$-supertiles according to the pattern of the substitution. 
For such tilings, 
the matrix $M_{n,N}$ is just the $(N-n)$th power of the usual substitution 
matrix.
However, there is an important difference in perspective between substitutions
and fusions. 

A substitution can be viewed as a map from a tiling space to itself,
in which all tiles are enlarged and then broken into smaller
pieces. This map can be repeated indefinitely. In a fusion tiling, we
can likewise break each $n$-fusion tile into level $(n-1)$-supertiles
using the {\em subdivision map} $\sigma_n$, which is a map from
$X^{n}_\rrr$ to $X^{n-1}_\rrr$.  Unlike the substitution map for
self-affine tilings, it cannot go from $X_\rrr$ to itself, and this
map cannot be repeated more than $n$ times. Once you are down to the
atomic level (i.e., ordinary tiles), you cannot subdivide further!
The proofs of theorems about substitution tilings often involve taking
an arbitrary tiling and applying a substitution, or sometimes its
inverse, enough times to achieve a desirable result. For general
fusion tilings, this line of reasoning usually does not work.

\subsubsection{Induced fusions} Let $\{N(n)\}_{n = 1}^\infty$ be an
increasing sequence of positive integers.  The {\em induced fusion on
  $N(n)$ levels}, $\rrr^{ind}$, is obtained from a given fusion $\rrr$
by composing the fusions for levels $N(n)+ 1, ..., N(n+1)$ into one
step. In this case the supertiles of $\rrr^{ind}$ are given by
$\ppp_n^{ind} = \ppp_{N(n)}$, where the $N(n)$-supertiles are seen as
fusions of $N(n-1)$-supertiles.  The transition matrices for
$\rrr^{ind}$ are given by $M_{n,n+1}^{ind} = M_{N(n),N(n+1)}$.

\subsubsection{All FLC tilings are fusion tilings}
\label{alltilingsarefusions}
It is possible to view {\em any} tiling $\T$ of $\R^d$ from a given
prototile set $\ppp_0$ as a fusion tiling, as long as it has finite
local complexity.  Let the set $\ppp_n$ consist of all connected
patches containing $n$ tiles or less.  By finite local
complexity this is a finite
set. Each element of $\ppp_n$ is either an element of $\ppp_{n-1}$ or
is the fusion of two elements of $\ppp_{n-1}$.  (The fact that these
fusions typically are not unique does not matter).

\subsection{Common assumptions}
The previous section shows that the category of fusion tilings is
extremely general. To prove meaningful results, we have to impose
additional conditions on our fusion rules.  We collect several of them
into this section.

\subsubsection{Prototile- and transition-regularity}
These are the cases that are most similar to the usual definitions of
symbolic and tiling substitution.  When the number of supertiles at
each level is constant, we can associate each $n$-supertile to a
specific prototile, regardless of whether there is a geometric connection
between the two.
When we do this we call the fusion rule {\em
 prototile-regular} and rewrite it as:
$$\rrr= \left\{ P_n(p) \,\, | \,\,n \in \N  \text{ and } 
p \in \ppp_0 \right\}.$$

If the number $j_n$ of supertiles at the $n$th level of a fusion rule
$\rrr$ has $J = \liminf j_n$ for some finite $J$, then the fusion rule
is equivalent to a prototile-regular fusion rule by inducing on the
levels for which $j_n = J$.  The price we pay for taking such an
induced fusion is that the transition matrices can become wildly
unbounded. 

In the special case where the number of supertiles at each level is a
fixed constant $J$, if the transition matrices are all equal to a
single matrix we call the fusion rule a {\em transition-regular}
fusion rule.  Being transition-regular is considerably stronger than
being prototile-regular.  All substitution sequences and self-affine
tilings as defined in, for instance,
\cite{Robinson.ams,Sol.self.similar} are transition-regular, but not
every transition-regular fusion tiling comes from a substitution.  The
combinatorics and geometry of how the $(n-1)$-supertiles join to form
$n$-supertiles can change from level to level.

\begin{ex}{\em A fusion that is transition-regular but not a substitution.}
\label{not.substitution}
Consider a 1-dimensional fusion rule with transition matrix
  $\left ( \begin{smallmatrix} 2&1 \cr 1&2 \end{smallmatrix} \right )$
  in which $P_n(a)$ is always given by the word
  $P_{n-1}(a)P_{n-1}(a)P_{n-1}(b)$, and in which $P_n(b)$ is given by
  $P_{n-1}(b)P_{n-1}(b)P_{n-1}(a)$ if $n$ is prime, and is given by
  $P_{n-1}(a)P_{n-1}(b)P_{n-1}(b)$ if $n$ is composite.
\end{ex}

\noindent{\bf Remarks.} \begin{enumerate}  
\item Pseudo-self-similar (or self-affine) tilings, such as the
  Penrose tiling with kites and darts, are also transition-regular
  fusion tilings. In many cases these are asymptotically self-similar,
  and this asymptotic structure was used \cite{Me.Boris, BG} 
   to show that such tilings are topologically equivalent
  to self-similar tilings with fractal boundaries.
\item In the correspondence between one-dimensional
fusion tilings with $G=\Z$ and Bratteli-Vershik systems, 
prototile-regular tilings correspond to finite Bratteli diagrams.   The finite list
of vertices on the $n$th level of the Bratteli diagram represents the  finite
set of $n$-supertiles.
\item The one-dimensional $S$-adic substitution sequences of Durand 
\cite{Durand}
  can be recast as fusion tilings, as can the linearly recurrent
  Delone sets and tower systems in \cite{BBG,APC}. Example
\ref{not.substitution} is $S$-adic.
\item The ``non-constructive'' combinatorial substitutions in \cite{My.primer}
 are exactly the class of prototile-regular fusion tilings.
\end{enumerate}

\subsubsection{Primitivity}

A fusion rule is said to be {\em primitive} if, for each non-negative
integer $n$, there exists another integer $N$ such that every
$n$-fusion supertile is contained in every $N$-supertile.  When the
fusion rule is transition-regular this is equivalent to some power of
the transition matrix having strictly positive entries.  In general it
is equivalent to there existing an $N$ for each $n$ such that $M_{n,N}$ has 
all positive entries. A fusion rule is called {\em strongly
  primitive} if for every $n \ge 1$, each $(n+1)$-supertile contains at
least one copy of every $n$-supertile.  That is, all of the
transition matrices $M_{n,n+1}$ have strictly positive entries.  Any
primitive fusion rule is equivalent to a strongly primitive one by
inducing on enough levels.

Primitivity is one of the most common assumptions used in the literature on
substitution sequences and tilings.   It allows for Perron-Frobenius theory
to be applied to the systems to determine natural frequencies, volumes,
and expansion rates.    We will adapt  this analysis to the fusion 
situation in Section \ref{Dynamics.section}.

\subsubsection{Recognizability} 
%
A fusion rule $\rrr$ is said to be {\em recognizable} if, for each
$n$, the subdivision map $\sigma_n$ from $X^n_\rrr$ to $X^{n-1}_\rrr$ is
a homeomorphism. If so, then every tiling in $X_\rrr$ can be
unambiguously expressed as a tiling with $n$-supertiles for every $n$.
The uniform continuity of the inverse subdivision maps then implies that
there exists a family of {\em recognizability radii} $r_n$
($n=1,2,\ldots$), such that, whenever two tilings in $X_\rrr$ have the
same patch of radius $r_n$ around a point $\vecv \in \R^d$, then the
$n$-supertiles intersecting $\vecv$ in those two tilings are
identical.

For substitution sequences and tilings, recognizability is closely
related to non-periodicity \cite{Mosse,Sol.u.comp}. Recognizability
implies that none of the tilings are periodic. Conversely, if $G$
consists only of translations \cite{Sol.u.comp}, or if $G$ contains a
set of rotation about the origin with no invariant subspaces then the
absence of periodic tilings in $X_\rrr$ implies recognizability
\cite{HRS}.  However, it is easy to construct fusion rules that are
nonperiodic but not recognizable. For instance, the Fibonacci tiling can
be generated either from the substitution $a \to ab$, $b \to a$ or from
the substitution $a \to ba$, $b \to a$. By including both sets of supertiles
in our fusion rule, we obtain a description of the non-periodic 
Fibonacci tiling space in which each tiling has at least two (actually more)
decompositions into $n$-supertiles for $n>0$.

We now show that fusion tiling spaces are topological factors of
recognizable fusion tiling spaces using a construction inspired by the
work of Robinson \cite{R.Robinson} and Mozes \cite{Mozes}.

\begin{ex} \label{recognize-ex}{\em Constructing a recognizable extension.}
Let $\rrr_0$ be a 1-dimensional fusion rule on the letters $a$ and $b$,
each of which is viewed as a tile of length 1.  If we let the $n$-supertiles
be all possible sequences of $a$s and $b$s of length $5^n$, then 
the space $X_{\rrr_0}$ is just the space of all bi-infinite tilings by 
$a$'s and $b$'s and $\rrr_0$ is clearly not recognizable. 

Now let $\rrr$ be 
a 1-dimensional fusion with four letters, $a_1$, $a_2$, $b_1$ and $b_2$.
We call $a_1$ and $b_1$ ``type 1'', and write $x_1$ to mean either
$a_1$ or $b_1$. Likewise $x_2$ means either $a_2$ or $b_2$.  
The 1-supertiles are all 5-letter words of the general form $x_2x_1x_1x_1x_1$
(where each $x_i$ denotes a separate choice of $a_i$ or $b_i$) 
or $x_2x_1x_2x_2x_1$. We will use $s^1_1$ are shorthand for supertiles 
of the first type and $s^1_2$ for the second. Note that each supertile
begins with an isolated $x_2$, and that isolated $x_2$'s appear only
at the beginning of supertiles. This makes the map from $X^1_\rrr$ to
$X_\rrr$ invertible. 

We repeat the coding at higher levels. Second-order supertiles can either
take the form $s^2_1 = s^1_2s^1_1s^1_1s^1_1s^1_1$ or 
$s^2_2 = s^1_2s^1_1s^1_2s^1_2s^1_1$, and generally $(n+1)$-supertiles 
can either take the form
$s^{n+1}_1 = s^n_2s^n_1s^n_1s^n_1s^n_1$ or 
$s^{n+1}_2 = s^n_2s^n_1s^n_2s^n_2s^n_1$. By the same reasoning, all decomposition
maps are invertible, and $\rrr$ is recognizable. Finally, the factor map
$X_\rrr \to X_{\rrr_0}$ just erases the subscripts on all of the letters. 
\end{ex}

The details of the construction will be different for different
examples and can get complicated if the supertiles have wild shapes
or combinatorics, but the basic idea is universal. Pick sufficiently
many copies of your original tile set.  Use some of the labels within
a first-order supertile to indicate which tiles are in the supertile,
and the rest to give the first-order supertiles labels. Use some of
those first-order labels to define the boundaries of the second-order
supertiles, and the rest to label the second-order
supertiles. By continuing the process ad infinitum we obtain a recognizable
fusion tiling space that factors onto the original.  How close to
 this factor map is to being one-to-one becomes an important question.

\subsubsection{Van Hove sequences and fusion rules}
\label{VanHove-sec}
A van Hove sequence $\{A_m\}$ of subsets of $\R^d$ consists of sets
whose boundaries are increasingly trivial relative to their interiors
in a precise sense.  In many cases it will be convenient to consider
only fusion rules where the supertiles share this property. The use of
van Hove sequences, which for $\R^d$ is equivalent to F\o lner
sequences, is adopted from statistical mechanics.  We follow the
notation of \cite{Sol.self.similar} here and define, for any set $A
\in \R^d$ and $r > 0$:
$$ A^{+r} = \{x\in \R^d : \text{ dist}(x, A) \le r\}, $$
where ``dist'' denotes Euclidean distance. A sequence of sets $\{A_n\}$ of sets
in $\R^d$ is called a {\em van Hove} sequence if for any $r \ge 0$
$$\lim_{ n \to \infty} \frac{\text{Vol}\left((\partial A_n)^{+r}\right)}
{\text{Vol}(A_n)} 
= 0,$$ 
where $\partial A$ is the boundary of $A$ and $\text{Vol}$ is Euclidean
volume. 

Given a fusion rule $\rrr$, we may make a sequence of sets in $\R^d$
by taking one $n$-supertile for each $n$ and calling its support
$A_n$.  We say $\rrr$ is a {\em van Hove fusion rule} if every such
sequence is a van Hove sequence. Equivalently, a fusion rule is van Hove
if for each $\epsilon >0$ and each $r>0$ 
there exists an integer $n_0$ such that each 
$n$-supertile $A$, with $n \ge n_0$, has $\text{Vol}(\partial A)^{+r}<
\epsilon \text{Vol}(A)$. 

\subsection{Notational conventions}

Entries of vectors and matrices are indicated as arguments, while
subscripts are used to distinguish between different vectors and
matrices. Thus, $M_{1,2}(3,4)$ is the (3,4) entry of the matrix
$M_{1,2}$ and $v_5(2)$ is the second entry of $v_5$. Vectors are
viewed as columns, so that the product $Mv$ of a matrix and a vector
is well-defined.  Groups are denoted by capital letters, as are
subsets of groups, while elements of groups are lower
case. Collections of patches of tilings are given by calligraphic
letters $\ppp, \rrr$, etc, and in particular our fusion rules are so
denoted. Tilings are bold face. Elements of physical space $\R^d$
are marked with arrows, and the dot product is
reserved for this setting.

\section{Dynamics of fusion tilings}
\label{Dynamics.section}

Let $G = G_t \rtimes G_r$, where $G_t$ is the
translation subgroup and $G_r$ is the point group $G/ G_t$.  By
assumption $G$ contains a full rank lattice of translations, and $G_r$
is a closed subgroup of $O(n)$.  
Let $Vol$ be
Haar measure on $G_t$, a product of Lebesgue measure in
the continuous directions of $G_t$ and counting measure in the discrete
directions, and let $\lambda_0$ be normalized Haar measure on
$G_r$. Let $\lambda$ be a measure on $G$ with
$\lambda(U_t \rtimes U_r) = Vol(U_t) \lambda_0(U_r)$ for every pair of 
measurable sets $U_t \in G_t$, $U_r \in G_r$.  We 
assume that  we have a metric on $G$ whose
restriction to $G_t$ is Euclidean distance and whose restriction to
$G_r$ is bounded by 1.

\subsection{Tiling metric topology and dynamical system}
\label{metric-sec}
Let $X_\ppp$ be the set of all possible tilings using some fixed
prototile set $\ppp$ and some group $G$ of isometries.  (That is, 
a point in this space is an entire tiling of $\R^d$.)  We turn
$X_\ppp$ into a metric space using the so-called ``big ball'' metric
using the metric on $G$ as follows. Two tilings $\T_1$ and $\T_2$ are
$\epsilon$-close if there exist group elements $g_1$ and $g_2$, each
of size less than or equal to $\epsilon$, such that $g_1(\T_1)$ and
$g_2(\T_2)$ exactly agree on the ball of radius $1/\epsilon$ around the
origin.

This metric is not $G$-invariant, as it gives greatest weight to
points close to the origin, but the resulting topology {\em is\/} 
$G$-invariant. A sequence of tilings $\T_i$ converges to a tiling $\T$ if
there exists a sequence of group elements $g_i$, converging to the identity,
such that for every compact subset $K$ of Euclidean space, the tilings
$g_i(\T_i)$ eventually agree with $\T$ on $K$. 

%

\begin{dfn} Let $G' \subseteq G$ contain $G_t$ and let $X$ be a
  closed, $G'$-invariant subset of $X_\ppp$.  A {\em tiling dynamical
    system $(X,G')$} is the set $X$ together with the action of $G'$
  on $X$.
\end{dfn}

It is usually assumed in the tiling literature that $G' = G_t$.  This
can be assumed without loss of generality when $G_r$ is a finite group
simply by making extra copies of each prototile, one for each element
of $G_r$.  The situation is more complicated in cases such as the
pinwheel tiling, where $G_r$ is infinite.  

\subsection{Minimality} A topological dynamical system $(X, G')$ is
{\em minimal} if $X$ is the orbit closure of any of its elements.  

\begin{prop}
  If the fusion rule $\rrr$ is primitive, then the dynamical system
  $(X_\rrr, G)$ is minimal.
\end{prop}
  
\begin{proof}
  Let $\T \in X_\rrr$ be any fixed tiling.  We will show that given $\T'
  \in X_\rrr$ and $\epsilon > 0$ there is a group element $g$ such that
  $d(g(\T), \T') < \epsilon$.  Denote by
  $[\T']_r$ the patch of tiles in $\T'$ that intersect the ball of
  radius $r$ centered at the origin.  By definition we know
  that any such patch is admissible by $\rrr$ and so there is an $n \in
  \N$ and a $i \in \{1, ..., j_n\}$ for which $[\T']_{1/\epsilon}$ is a
  subpatch of $P_n(i)$.

  On the other hand, primitivity means that there is an $N$ such that
  every $N$-supertile contains a copy of
  $P_n(i)$.  Since $\T$ is a union of $N$-supertiles, it
  contains many copies of $P_n(i)$.  Pick $g$ to 
  bring some particular copy of $P_n(i)$ to the
  origin in agreement with $[\T']_{1/\epsilon}$.  Since $\T'$ and $g(\T)$
  are identical on the ball of radius $1/\epsilon$ about
  the origin, $d(g(\T), \T') < \epsilon$.
\end{proof}
\noindent{\bf Remarks.} \begin{enumerate}  
\item It is not necessarily true that $(X_\rrr,G_t)$, i.e. the
  dynamical system with only translations acting, is minimal.
  Consider any fusion rule having only finitely many relative
  orientations of the prototiles, but which for some reason we took
  $G$ to be the full Euclidean group. In this case $(X,G)$ would be
  minimal but $(X_\rrr,G_t)$ would not.  No tiling could approximate
  an irrational rotation of itself.
\item On the other hand, the pinwheel tilings \cite{Radin.annals}
  provide an example where $G$ is the full Euclidean group but $(X,
  G_t)$ is minimal.
\item Primitivity is sufficient but not necessary for minimality. In
  particular, the Chacon transformation is not primitive, but is
  minimal. For each $n$ there does not exist an $N$ for which $P_N(b)$
  contains $P_n(a)$. However, there does exists a radius $r_n$ such
  that every ball of radius $n$ contains at least one $P_n(a)$ and at
  least one $P_n(b)$, so the patch $[\T']_{1/\epsilon}$ can be found in
  every $\T$.
\end{enumerate}


\subsection{Invariant measures in general tiling dynamical systems}
\label{frequency-sec}
We begin our treatment of the ergodic theory of fusion tilings with a
discussion of how invariant measures work for a general FLC tiling
dynamical systems $X$, with a focus on patch frequency.  For
convenience,  we assume for the remainder of Section \ref{Dynamics.section}
that our action is by $G = \R^d$ only. See section \ref{othergroups} 
for the modifications needed to apply this theory to other subgroups of the 
Euclidean group.   

Let $P$ be any patch of
tiles containing the origin.     
Let $U$ be a measurable subset of $\R^d$, let 
$X_{P,U}$ be the cylinder set
of all tilings in $X$ that contain $P - \vec v$ for some $\vec v \in U$, and
let $\chi_{P,U}$ to be the indicator function of this cylinder set.
The sets $X_{P,U}$, plus translates of these sets, 
generate our $\sigma$-algebra of measurable sets in
$X$.  Let $\mu$ be an invariant measure on $X$. 
If $U$ is sufficiently small, then for every tiling $\T \in X$, 
there is at most one $\vec v \in U$ 
for which $P - \vec v \subset \T$. Since the measure is additive and
translation-invariant, 
$\mu(X_{P,U})$ must be proportional to the volume of $U$ and we define
\begin{equation}\label{freq.def}
freq_\mu(P) =  \frac{1}{Vol(U)} \mu(X_{P,U}),
\end{equation}
a quantity that is independent of $U$.

For any $A \subset \R^d$ we denote the number of times an equivalent
copy of $P$ appears in $\T$, completely contained in the set $A$, as
$\#(P \text{ in } A \cap \T)$.  As a special case, if $P'$ is another
patch (usually some supertile), we denote by $\#(P \text{ in } P')$
the number of equivalent copies of $P$ completely contained in
$P'$. Next we pick a specific $U_0$ that is a small ball centered at the origin
and define the function
\begin{equation}
f_P(\T) = \frac{1}{Vol(U_0)} \chi_{P,U_0}(\T). 
\end{equation}
This is a smeared $\delta$-function that counts the appearances of $P$. 
$\int_A f_P(\T-\vec v) d\vec v$ is essentially $\#(P\text{ in }A
\cap \T)$, with small corrections for patches that come within the diameter
of $U_0$ of the boundary of $A$.  Note that 
$\int_X f_P(\T) d\mu = \frac{1}{Vol(U_0)} \mu(X_{P,U_0})=freq_\mu(P)$.

We use the following version of the pointwise ergodic theorem:
\begin{thm} \label{ergodic.from.tiling.freqs} Let $(X, \R^d)$ be a
  tiling dynamical system with invariant Borel probability measure
  $\mu$.  Let $\{A_m\}$ be a sequence of balls centered at the origin,
with radius going to infinity, and let $P$ be any finite
  patch.  Then for $\mu$-almost every tiling $\T$ the limit
\begin{equation} \label{invariantaverage}
\lim_{m \to \infty}\frac{1}{Vol(A_m)} \int_{A_m} f_P(\T-\vecv)
d\lambda(\vecv) = \bar f_P(\T)
\end{equation}
exists. Furthermore, $\int_X \bar f_P(\T) d\mu = \int_X f_P(\T) d\mu
= freq_\mu(P)$. If $\mu$ is ergodic, then for $\mu$-almost every $\T$,
$\bar f_P(\T) = freq_\mu(P)$. 
\end{thm}

The quantity $\bar f_P(\T)$ corresponds to the usual notion of 
frequency as ``number of occurrences per unit area'' in $\T$, as computed with
an expanding sequence of balls around the origin.\footnote{Ergodic theorems are
often stated not with balls, but 
in terms of F\o lner or van Hove sequences that have special
properties, such as being ``regular'' or ``tempered''. That generality is
useful for computing frequencies using different sampling regions,
or when considering more complicated groups than $\R^d$.
For our purposes, however, balls are sufficient.} We will use the term
``frequency'' for this quantity, and will call $freq_\mu(P)$ an ``abstract
frequency''. The ergodic theorem says
that almost all tilings have well-defined frequencies, and that
the abstract frequency $freq_\mu(P)$, 
while not necessarily the frequency of $P$ 
in any specific
tiling, is the average over all tilings of the frequency of $P$. Thus, 
any upper or lower bounds on the frequency of $P$ that apply to $\mu$-almost
every $\T$ result in upper or lower bounds on the abstract frequency. If every
tiling has the same frequency of $P$, then there is only one possible value for
the abstract frequency of $P$, and thus for the measure of any $X_{P,U}$.
Tiling spaces where all tilings have the same set of frequencies are
uniquely ergodic.  

For an FLC tiling, the set of all patches 
(up to translation) is countable, and the intersection
of a countable number of sets of full measure has full measure. 
As a result, $\mu$-almost every tiling $\T$ has well-defined frequencies for
{\em every} patch $P$. 

Conversely, if a tiling $\T$ has well-defined patch frequencies, then 
we can construct a probability measure on $X$ by taking
$\mu(X_{P,U}) = \bar f_P(\T) Vol(U)$ for small $U$ and extending by 
additivity to larger $U$'s.  The additivity properties of measures 
follow from the addititivity of frequencies. For instance, if a patch
$P$ can be extended in two ways, to $P'$ or $P''$, then 
$X_{P,U} = X_{P',U} \coprod X_{P'',U}$. The identity $\mu(X_{P,U})=
\mu(X_{P',U}) + \mu(X_{P'',U})$ follows from $\bar f_{P}(\T)
= \bar f_{P'}(\T) + \bar f_{P''}(\T)$. Countable additivity is not an issue,
since the tiling space is locally the product of Euclidean space (where
Lebesgue measure has all the desired properties) and a Cantor set (where
the $\sigma$-algebra is based on finite partitions into clopen sets). 

A measure defined in this way may or may not be
ergodic.  For instance, if $\T$ is a one-dimensional tiling with the
pattern $a^\infty b^\infty$, with $a$ tiles to the left of the origin
and $b$ tiles to the right, then the resulting measure on the orbit
closure of $\T$ is the average of the two ergodic measures.

\subsection{Invariant measures and fusion tilings}   
The possibilities for invariant measures of fusion tilings are
intimately connected to the asymptotic behavior of the transition
matrices $M_{n,N}$ as $N \to \infty$.  Our analysis of these matrices takes the
place of the standard Perron-Frobenius theory used so fruitfully for
substitution systems.  The results of this section and the 
next closely 
parallel those of \cite{Fisher, BKMS}, the difference being that
those papers consider discrete systems in 1 dimension, while we consider
continuous systems in an arbitrary number of dimensions. 
We assume that our fusion rule is van Hove,
recognizable, and primitive; these properties are essential.  
%
%

We define the frequency $\tilde f_{P_n(j)}$ of a supertile $P_n(j)$
in a tiling $\T$ to be its frequency {\em as a supertile}, not as a patch. 
In other
words, $\tilde f_{P_n(j)}(\T)$ 
is obtained by viewing $\T$ as an element of 
$X^n_\rrr$, thereby excluding patches that have the same composition
as $P_n(j)$, but are actually proper subsets of another $n$-supertile
or straddle two or more $n$-supertiles. The abstract supertile
frequency of $P_n(j)$ 
is obtained by averaging $\tilde f_{P_n(j)}$ over all tilings. 
By recognizability, each occurrence of a supertile $P_n(j)$ is marked by
an element of a set of larger patches $S_i$. We then have 
$\tilde f_{P_n(j)}(\T) = \sum_i \bar f_{S_i}(\T)$, and 
the abstract supertile frequency of $P_n(j)$ is $\sum_i freq_\mu(S_i)$.

Consider a sequence $\rho = \{\rho_n\}$ where each $\rho_n \in \R^{j_n}$ has
all nonnegative entries.  We say that $\rho$ is {\em
  volume-normalized} if for all $n$ we have $\sum_{i = 1}^{j_n} \rho_n(i)
Vol(P_n(i)) = 1$.  We say that it has {\em transition consistency} if
$\rho_n = M_{n,N} \rho_N$ whenever $n < N$.  A transition-consistent
sequence $\rho$ that is normalized by volume is called a sequence of
{\em well-defined supertile frequencies}.  This terminology will be
justified by the proof of Theorem \ref{measures.equiv.rns}.

\begin{thm} \label{measures.equiv.rns}
Let $\rrr$ be a recognizable, primitive, van Hove fusion rule.
There
is a one-to-one correspondence between the set of all invariant Borel
probability measures on $(X_\rrr, \R^d)$ and the set of all sequences
of well-defined supertile frequencies with the correspondence 
that, for all patches $P$, 
\begin{equation} \label{measure.from.frequency}
freq_\mu(P) = \lim_{n \to \infty} \sum_{i = 1}^{j_n} \#\left(P \text{
  in } P_n(i)\right) \rho_n(i)
\end{equation}
\end{thm}

\begin{proof}  Suppose $\mu$ is an invariant measure.   For each 
$n \in \N$ and each $i \in \{1, 2, ..., j_n\}$, define $\rho_n(i)$ to be
  the abstract supertile frequency of $P_n(i)$.  
For a fixed $n$, $X_\rrr$ is the union
  of cylinder sets given by which $n$-supertile is at the origin.
  Since the measure of $X_\rrr$ is 1 and the measure of each of these
  cylinder sets is $\rho_n(i) Vol(P_n(i))$, the sequence
  $\rho$ is volume-normalized.

The set of tilings where the origin lies in an $n$-supertile of type $i$ is
the union of disjoint sets where the origin lies in a supertile
of type $i$, which in turn sits in an particular way in an $N$-supertile. 
There are $M_{n,N}(i,j)$ ways for $P_n(i)$ to sit in $P_N(j)$.
The additivity of the measure implies that $\rho_n(i)
= \sum_{j = 1}^{j_N}M_{n,N}(i,j) \rho_N(j)$.  Hence 
$\mu$ gives rise to a sequence of well-defined supertile frequencies.

To see that 
equation (\ref{measure.from.frequency}) applies, let $P$ be any patch and call its
diameter $L_P$.  Since the fusion rule is van Hove, we can pick an $n$
large enough that the fraction of each $n$-supertile within $L_P$ of
the boundary is so small that $P$ patterns appearing in this region
can only contribute $\epsilon$ or less to the frequency of $P$'s in a
union of $n$-supertiles. 

To count the number of $P$'s in a large ball $A_m$ around the origin, 
we must count the
number of $P$'s in each $n$-supertile contained in that ball, plus the
number of $P$'s that straddle two of more $n$-supertiles,
plus the $P$'s in an $n$-supertile that is 
only partially in the ball. As a fraction of the whole, the third set goes
to zero as $m \to \infty$ and the second goes to zero as $n \to \infty$. 
Thus $\#(P \text{ in } A_m \cap \T) = \sum_{i = 1}^{j_n} \#(P \text{ in }
P_n(i)) \#(P_n(i) \text{ in } A_m) + $ boundary occurrences.  
Dividing by $Vol(A_m)$ and taking
limits, first as $m \to \infty$ and then as $n \to\infty$,
gives the identity
$$
\bar f_P(\T) = \lim_{n \to \infty} \sum_{i = 1}^{j_n} \#\left(P \text{
  in } P_n(i)\right) \tilde f_{P_n(i)}(\T)
$$
for all tilings $\T$ with well-defined patch frequencies.
Integrating this identity 
over all tilings then gives equation 
(\ref{measure.from.frequency}).

Now suppose that $\{\rho_n\}$ is a sequence of well-defined supertile
frequencies.  To establish
the existence of an invariant measure $\mu$ for which $\{\rho_n\}$
represents the abstract supertile frequencies, we simply define
$freq_\mu(P)$, and hence the measure of each cylinder set $X_{P,U}$,
by equation (\ref{measure.from.frequency}). 

To see convergence of 
the limit on the right hand side, note that, if
$n<N$, the number of $P$ in $P_N(j)$ is
the sum of the number of $P$ in each $n$-supertile in $P_N(j)$, plus
a small contribution from $P$'s that straddle two or more supertiles. 
That is,
$\#(P \text{ in } P_N(j))
\approx \sum_i \#(P \text{ in }P_n(i)) M_{n,N}(i,j)$, so 
$$\sum_{j} \#(P \text{ in } P_N(j)) \rho_N(j)
\approx \sum_{i,j} \#(P \text{ in }P_n(i)) M_{n,N}(i,j) \rho_N(j)
= \sum_{i} \#(P \text{ in } P_n(i)) \rho_n(i).$$
As $n \to \infty$ 
the contribution of $P$'s that straddle two $n$-supertiles goes to
zero, so the right hand side of 
(\ref{measure.from.frequency}) is a Cauchy sequence. 

The non-negativity of the measure follows 
from the non-negativity of each $\rho_n$. The identity $\mu(X_\rrr)=1$
follows from volume normalization. Finite 
additivity follows from the observation
that, if a patch $P$ can be extended to $P'$ or $P''$, then
$\#(P \text{ in }P_n(i)) = \#(P' \text{ in }P_n(i)) + \#(P'' \text { in }
P_n(i))$, plus a small correction for the situations where $P$ is completely
contained in $P_n(i)$ but $P'$ or $P''$ is not, a correction that does
not affect the limit as $n \to \infty$. As noted earlier, countable 
additivity is not an issue for tiling spaces. Thus $\mu$ is a well-defined
measure. 

\end{proof}

\subsubsection{Parameterization of invariant measures.}
In Theorem \ref{measures.equiv.rns} we showed how measures relate to
well-defined sequences of supertile frequencies. We now
give an explicit parametrization of the invariant measures in terms
of the transition matrices $M_{n,N}$, a parametrization that we will
then use to investigate unique ergodicity.

The {\em direction} of each
column of $M_{n,N}$ is defined to be the volume-normalized vector in $\R^{j_n}$ collinear
with it, and we define the {\em direction matrix} $D_{n,N}$ to be the
matrix whose columns are the directions of the columns
of $M_{n,N}$.  That is, 
$$D_{n,N}(*,k) = \frac{M_{n,N}(*,k)}{\sum_{l = 1}^{j_n} M_{n,N}(l,k)Vol(P_n(l))}.$$ 
%

Let $\poly_{n,N}$ be the convex hull of the
columns of $D_{n,N}$, sitting within the hyperplane of volume-normalized  
vectors in $\R^{j_n}$. 
Note that the extreme points of $\poly_{n,N}$ are columns of 
$D_{n,N}$, but not every column need be an extreme point. 
Since each column of $M_{n,N+1}$ is a sum of columns of $M_{n,N}$,
each column of $D_{n,N+1}$ is a weighted average 
of the columns of $D_{n,N}$, so $\poly_{n,N+1} \subset \poly_{n,N}$. 
Let $\poly_n = \cap_{N=n+1}^\infty \poly_{n,N}$.   

The matrix $M_{n,N}$ defines an affine map sending $\poly_N$
to $\poly_n$, since if $\rho_N$ is volume-normalized in $\R^{j_N}$, then so 
is $M_{n,N} \rho_N \in \R^{j_n}$. 
We define $\poly_\infty$ to be the inverse limit of the polytopes
$\poly_n$ under these maps. 

\begin{cor} \label{poly.parameterizes} Let $(X_\rrr, \R^d)$ be the
  dynamical system of a recognizable, primitive, van Hove fusion rule.
  The set of all invariant Borel probability measures is parameterized
  by $\poly_\infty$.
\label{parameterize.measures}
\end{cor}

\begin{proof}
  By Theorem \ref{measures.equiv.rns}, we need only show that each
  element of $\poly_\infty$ gives rise to a sequence $\{\rho_n\}$ of
  well-defined supertile frequencies and vice versa.

By construction, each point in $\poly_\infty$ is a sequence of 
well-defined supertile frequencies, since each point in $\Delta_n$ is
volume-normalized and non-negative, and since the sequence has
transition consistency. For the converse, suppose that $\{ \rho_n\}$
is a sequence of well-defined supertile frequencies. We must show that  
$\rho_n \in \poly_n$. Since $\rho_n = M_{n,N} \rho_N$, $\rho_n$ is 
a non-negative linear combination of the columns of $M_{n,N}$, and so is
a weighted average of the columns of $D_{n,N}$. Thus $\rho_n \in 
\poly_{n,N}$. Since this is true for every $N$, $\rho_n \in \poly_n$.
\end{proof}

\subsubsection{Measures arising from supertile sequences}
In this section we provide a concrete way of visualizing certain invariant
measures, in particular the ergodic ones.   The way to do it is by looking
at frequencies of patches as they occur in specific sequences of
nested supertiles.

\begin{dfn}
Let $\kappa = \{k_n\}$ be a sequence of supertile labels, 
with $k_n \in \{1,2,\ldots,j_n\}$. For each $n<N$, we consider the 
frequency of each $n$-supertile $P_n(i)$ within $P_N(k_N)$:
$$\rho_{n,N}(i)= M_{n,N}(i,k_N)/Vol(P_N(k_N)).$$
We say that $\kappa$ has {\em well-defined supertile
frequencies} if $\rho_n(i) = \lim_{N \to \infty} \rho_{n,N}(i)$ exists
for every $n$ and every $i \in \{1,\ldots, j_n\}$.
\end{dfn}

Note that the vectors $\rho_n(i)$, if they exist, do indeed form a
sequence of well-defined supertile frequencies. For $n<n'<N$,
$\rho_{n,N} = M_{n,n'}\rho_{{n'}\!,N}$. Taking a limit as $N \to \infty$ gives
$\rho_n = M_{n,n'} \rho_{n'}$, so the sequence has transition consistency.  
Likewise, it is easy to check volume normalization. 
We can therefore associate an invariant measure to every sequence 
$\kappa$ that has well-defined supertile frequencies.

The purpose of using a sequence $\kappa$ is to visualize a
measure.  Given such a sequence, one can imagine a tiling $\T$ where
the origin sits inside a $k_1$ 1-supertile, which sits inside a $k_2$
2-supertile, etc.  Under mild assumptions, the supertile 
frequencies $\tilde f_{P_n(i)}(\T)$ will then equal $\rho_n(i)$, and 
for any patch $P$, $\bar f_P(\T)$ will equal $freq_\mu(P)$, where 
$\mu$ is constructed from the sequence $\{\rho_n\}$.
The concept of using sequences $\kappa$ to
obtain measures applies even to non-primitive fusions, as long as the
supertile frequencies are well-defined.

\begin{ex}{\em A minimal fusion rule with two ergodic measures.}
This is a variation on an example found in \cite{FFT} and illustrates
the results of \cite{Down}.
Consider a prototile-regular 1-dimensional fusion rule with two unit
  length tiles $a$ and $b$ and let $G = \R$. Let $P_n(a) =
  (P_{n-1}(a))^{10^n} P_{n-1}(b)$ and $P_n(b) = (P_{n-1}(b))^{10^n}
  P_{n-1}(a)$, so that $P_1(a) = aaaaaaaaaab$ and $P_1(b) =
  bbbbbbbbbba$, etc.  
  $M_{n-1,n}=\left ( \begin{smallmatrix} 10^n & 1 \\ 1 & 10^n 
\end{smallmatrix} \right )$
%
%
%
which has eigenvalues $10^n - 1$ and
  $10^n + 1$.  Elementary linear algebra allows us to
  compute the frequencies as follows.  Let $\alpha_n = \prod_{k=1}^n
  \frac{10^k-1}{10^k+1}$, which approaches a limit of just over 0.8 as
  $n \to \infty$.  The fraction of $a$'s in $P_n(a)$ is
  $(1+\alpha_n)/2 \approx 0.9$, while the fraction of $a$'s in
  $P_n(b)$ is $(1-\alpha_n)/2 \approx 0.1$. 
 
  There are exactly two ergodic measures on this system.
$\Delta_n$ is an interval for every value of $n$,
with endpoints defined by the limits of the first and second columns of 
$D_{n,N}$. Likewise, $\Delta_\infty$ is an interval, whose endpoints
$\mu_a$ and $\mu_b$ can be obtained from the supertile sequences 
$\kappa=(a,a,a,a,\ldots)$ and 
$\kappa=(b,b,b,b,\ldots)$. 
  The first ergodic measure, $\mu_a$, sees the frequencies of $a$'s and
  $b$'s as measured in the type-$a$ supertiles and thus is $a$-heavy;
  the second, $\mu_b$, reverses the roles of $a$ and $b$ and is
  $b$-heavy.  The measure $\mu =
  (\mu_a + \mu_b)/2$ is invariant but not ergodic; this measure
  corresponds to Lebesgue measure when the system is
  seen as a cut-and-stack transformation.

\label{not-uniquely-ergodic-ex}\end{ex}

In general, a prototile-regular fusion with $j$ species of tiles can
have at most $j$ ergodic measures. Of course, there can be fewer, if one
or more columns of $D_{n,N}$ are in the convex hull of the others for large
$N$. The following example shows how a sequence $\kappa$ may lead to a 
measure that is not ergodic. 

\begin{ex}{\em A non-ergodic measure coming from a sequence $\kappa$.}
\label{non-ergodic.ex}
Consider the following variant of the previous example. Instead of having
two species of tiles or supertiles, we have three, with the fusion rules\\
$P_n(a) = (P_{n-1}(a))^{10^n}P_{n-1}(b)P_{n-1}(c)(P_{n-1}(a))^{10^n}$, \\
$P_n(b) = (P_{n-1}(b))^{10^n}P_{n-1}(a)P_{n-1}(c)(P_{n-1}(b))^{10^n}$, \\
$P_n(c) = (P_{n-1}(a))^{10^n}P_{n-1}(c)P_{n-1}(c)(P_{n-1}(b))^{10^n}$, \\
with
transition matrix $M_{n-1,n}=\left ( \begin{smallmatrix}
2\times 10^n & 1 & 10^n \cr 1&2\times 10^n & 10^n \cr
1 & 1& 2 \end{smallmatrix} \right )$. As before, the measure $\mu_a$
coming  from the sequence $a,a,\ldots$ is ergodic and describes the patterns in
a high-order $a$ supertile, which is rich in $a$ tiles, 
while the measure $\mu_b$ describes the
patterns in a high-order $b$ supertile, which is similarly rich in $b$
tiles. The measure $\mu_c$ from $c,c,\ldots$ describes a high-order
$c$ supertile, which is (essentially) half high-order $a$ supertiles and half 
high-order $b$ supertiles. In other words, $\mu_c = (\mu_a+\mu_b)/2$ is
not ergodic.
\label{kappa-ex}\end{ex}

\subsection{Unique ergodicity}
A system is {\em uniquely ergodic} if it has exactly one ergodic
probability measure, in which case this measure is the only invariant
probability measure whatsoever.  Tiling dynamical
systems are uniquely ergodic when there are uniform patch frequencies
that can be computed regardless of the tiling
(see e.g. Theorem 3.3 of \cite{Sol.self.similar}).

For each $n$, we say that the family of matrices $D_{n,N}$
is {\em asymptotically rank 1} if there is a vector $d_n \in \R^{j_n}$
such that the columns of $D_{n,N}$ all approach 
$d_n$ as $N \to \infty$. Put another way, $D_{n,N}$ is asymptotically
rank one if $\poly_n$ consists of a single point. 

\begin{thm}\label{unique}
  If a primitive fusion rule $\rrr$ is van Hove and recognizable, then $D_{n,N}$
  is asymptotically rank 1 for every $n$ if and only if the tiling
  dynamical system $(X_\rrr, \R^d)$ is uniquely ergodic.
\end{thm}

\begin{proof}
  By Corollary \ref{poly.parameterizes}, having a unique measure is
  the same as $\poly_\infty$ being a single point, which is equivalent
  to each $\poly_n$ being a single point.
\end{proof}
 
\begin{cor}
  The tiling dynamical system of a transition-regular fusion rule that
  is recognizable, van Hove and primitive is uniquely ergodic.
\end{cor}

What remains is to find checkable conditions on the transition
matrices $M_{n,N}$ that imply that the direction matrices $D_{n,N}$
are asymptotically rank one.  For the $n$-th transition matrix
$M_{n-1,n}$, let
\begin{equation} \delta_n = \min_k \left ( \frac{\min_i 
M_{n-1,n}(i,k)}{\max_i M_{n-1,n}(i,k)}
\right )
\label{checkable_condition}
\end{equation}
This measures the extent to which the columns of $M_{n-1,n}$ are
unbalanced.

\begin{thm} \label{unique0}
If $\sum_n \delta_n$ diverges, then $\rrr$ is primitive and
for each $n$ the family $D_{n,N}$
is asymptotically rank 1. 
\end{thm}


\begin{proof} 
First we show that the diameter of $\poly_{n,N+1}$ is bounded 
by $( 1 - \delta_{N+1} )$
times the diameter of $\poly_{n,N}$.
Let $v_{n,N}$ be the sum of the columns of $M_{n,N}$, and let
$\hat v_{n,N}$ be the direction of $v_{n,N}$. Let $m_{N+1,i}$ be the
smallest entry of the $i$th column of $M_{N, N+1}$ (which may be
zero).  The $i$th column of $M_{N,N+1}$ is then $m_{N+1,i}\left
  ( \begin{smallmatrix} 1 \cr \vdots \cr 1 \end{smallmatrix} \right
)$, plus additional terms, so the $i$th column of $M_{n,N+1}$ is
$m_{N+1,i} v_{n,N}$, plus an additional linear combination of
columns of $M_{n,N}$. This means that the direction of the $i$th
column of $M_{n,N+1}$ is a weighted average of $\hat v_{n,N}$ and an
unknown element of $\poly_{n,N}$, with $\hat v_{n,N}$ having weight at
least $\delta_{N+1}$. Thus the direction of each column of
$M_{n,N+1}$, and hence $\poly_{n,N+1}$ lies in the convex set
$\delta_{N+1} \hat v_{n,N}+ (1 - \delta_{N+1})\poly_{n,N}$, a set whose
diameter is $(1 - \delta_{N+1})$ times the diameter of $\poly_{n,N}$.

If $\sum \delta_n$ diverges, then $\delta_n$ is nonzero
infinitely often, so the fusion rule is primitive.
Furthermore, the infinite product $\prod_{k=n+1}^\infty
(1 - \delta_k)$ equals zero. Thus
$\poly_n$ has diameter zero, and is a single point. 
\end{proof}

\begin{cor}\label{primitive.ergodic}
If $\rrr$ is a strongly primitive, van Hove and recognizable fusion rule whose transition matrices
$M_{n-1,n}$ have uniformly bounded elements, then $(X_\rrr, \R^d)$ is 
uniquely ergodic.
\end{cor}

\begin{proof} If the smallest matrix element of $M_{n-1,n}$ is at least 1 and
the largest is at most $K$, then each $\delta_n \ge 1/K$. Thus 
$\sum_n \delta_n$ 
diverges and every $D_{n,N}$ is asymptotically rank 1.  By Theorem \ref{unique},
$(X_\rrr,\R^d)$ is uniquely ergodic. \end{proof}

\subsection{Transversals, towers, and rank}
In tiling theory, especially the aperiodic order and quasicrystal
branches, the concept of the transversal is an essential component to
many arguments.  For instance, it is used for computing the
$C^*$-algebras and $K$-theory as in \cite{Kel-Put} and references
therein, and it is used for gap-labelling in \cite{BBG}.  In ergodic
theory, the concept of towers and especially the Rohlin Lemma (also
called the Kakutani-Rokhlin or Halmos-Rokhlin Lemma) is a tool that
has been used to great effect (see for instance
\cite{Pet,Ornstein-Weiss}) .  One notable result \cite{AOW} that uses
towers and the lemma is that any aperiodic measure-preserving
transformation on a standard Lebesgue space can be realized as a
cutting and stacking transformation.  Towers are used to define the
notion of rank, which is intimately related to spectral theory.

For convenience, we will assume that $G = G_t$ and that our supertile
sets $\ppp_n$ are described as follows. We position each prototile in
$\ppp_0$ so that the origin is in its interior, and the place where
the origin sits is called the {\em control point} of the prototile.
In a prototile that has been translated by some element $\vec v \in
\R^d$ we call $\vec v$ the {\em control point} of the new tile.  Each
element of $\ppp_1$ is positioned such that the origin is on the
control point of one of the prototiles it contains, and this point
will be considered to be the control point of the 1-supertile.
Likewise, we situate the elements of $\ppp_2, \ppp_3,$ etc. in such a
way that the origin forms the control point of the $n$-supertile and
lies atop the control point of the $\ppp_{n-1}$-tile at the origin.

\begin{dfn}
  The {\em transversal} of $X_{\rrr}$ is the set of all tilings
  positioned with the origin at the control point of the tile that
  contains it.  If $\rrr$ is recognizable, the {\em $n$-transversal}
  of $X_{\rrr}$ is the set of all tilings positioned so that the
  origin is at the control point of the $n$-supertile containing it.
\end{dfn}

(If $G_r$ is nontrivial, the situation is only slightly more complicated.
For each prototile, we fix a
preferred orientation. The transversal of a tiling space is the set of
tilings with the origin at a control point, and with the tile
containing the origin in the chosen orientation. The $n$-transversals
are defined similarly.)

The transversal of $X_\rrr$ has a natural partition into
$j_0$ disjoint sets, one for each type of tile. Each of these can be
decomposed into pieces, one for each way that the tile containing the origin
can sit in a 1-supertile, and this partitioning process can be continued
indefinitely. 
The $n$-transversal of $X_{\rrr}$ can be thought of as the
transversal of $X_{\rrr}^n$, and can likewise be partitioned.  
When the fusion rule has finite local
complexity, the transversal is a totally disconnected set.  The $n$-transversals
will form the {\em base} for the $n$th tower representation.

In one-dimensional discrete dynamical systems, phase space can be visualized 
as a stack of Borel sets placed one above the other with
the transformation taking each set to the one directly above it,
except the top one, on which the action is not visualized.  This
representation of the system is known as a Rohlin tower.
When the action is continuous,
multidimensional, or by an unusual group, the ``towers'' no longer
resemble physical towers, but the term still applies. 
The concept of Rohlin towers for groups other than $\Z$,
and in particular for $\R^d$, is
investigated in \cite{Ornstein-Weiss,
  Aimee-Ayse,Danilenko-Silva,Robbie-Ayse} and our definitions are
drawn from these.  Let $(X,\mathfrak{B},\mu)$ be a probability space
acted on by some amenable group $G$ to produce an ergodic dynamical system.

\begin{dfn}
\begin{enumerate}
\item  Let $B \subset \mathfrak{B}$ and let $F \subset \R^d$ be relatively
  compact, and suppose that $g(B) \cap h(B) = \emptyset$ for any $g
  \neq h \in F$.  In this case we call $(B,F)$ a {\em Rohlin tower}
  with {\em base} $B$, {\em shape} $F$, and {\em levels} $g(B)$.
\item A {\em tower
    system} is a finite list of towers $\mathcal{F}= (B_1, F_1), ..., (B_n,F_n)$
  such that all levels are pairwise disjoint.  
 \item The {\em support} of a tower system is the union of its levels and the {\em
 residual set} is the complement of the support in $X$.
 \item
  A sequence $\mathcal{F}^k$ of tower
  systems is said to converge to $\mathfrak{B}$ if for every Borel set
  $A \in \mathfrak{B}$ and every $\epsilon > 0$ there is an $N$ such
  that for all $k > N$ 
there is a union of levels of $\mathcal{F}^k$ whose symmetric difference
from $A$ measures less than $\epsilon$.
\end{enumerate}
\end{dfn}

Recognizable fusion tiling dynamical systems come automatically
equipped with tower systems that converge to the Borel
$\sigma$-algebra $\mathfrak{B}(X_{\rrr})$.  The $n$th tower system
will have one tower for each prototile $P_n(j) \in \ppp_n$, for a
total of $j_n$ towers.  The base of the $j$th tower is the set
$B_n(j)$ of all tilings in the $n$-transversal that have a copy of
$P_n(j)$ with its control point at the origin.  The shape of the $j$th
tower, denoted $F_n(j)$, depends on whether $G_r$ is trivial.  If so,
then the shape is the set of all $\vec v \in G_t$ such that the origin
is in the interior of $P_n(j) + \vec v$.  This shape is, up to sign,
the same as the interior of $P_n(j)$ itself, thus earning the name
``shape''.  If instead $G_r$ is nontrivial, the tower construction
must be modified to accomodate tilings that have discrete rotational
symmetry.  In this case $F_n(j)$ is the set of all $g \in G$ for which
the origin is in $g(P_n(j))$, provided $P_n(j)$ has no symmetry.  If
it does, we must restrict the rotational portion of $F_n(j)$ to keep
the levels disjoint.

In ergodic theory an important idea is that of rank.  A
dynamical system is said to have {\em rank $r$} if for every $\epsilon>0$
there is a tower system $(B_1, F_1), ..., (B_r,F_r)$ that approximates
all elements of $\mathfrak{B}$ up to measure $\epsilon$, where $r$ is
the smallest integer for which this is possible.  It is well-known for
substitution sequences and self-affine tilings that the rank is
bounded by the size of the alphabet or the number of prototiles, since
every tile type gets its own tower for each application of the
substitution.  In the case of fusion tilings, the situation is only
slightly more complicated and we can say
\begin{equation}
\text{ rank}(X_{\rrr},G) \,\, \le \,\, \liminf_{n \to \infty} j_n
\end{equation}
In general, rank bounds spectral multiplicity.

\subsection{Groups other than $\R^d$}\label{othergroups}
For much of this section we have assumed that $G=\R^d$. However, the results
can readily be adapted to tiling spaces that involve other groups. 
In this section
we indicate what changes have to be made when $G_t$ is a proper subgroup
of $\R^d$, when $G_r$ is nontrivial, or both. 

In general, $G_t$ is the product of two groups, namely a continuous 
translation in a subspace $E$ of $\R^d$, and a discrete lattice $L$ in the
orthogonal complement of $E$.  In place of Lebesgue measure on $\R^d$,
the measure on $G_t$ is the product of Lebesgue measure on $E$ 
and counting measure on $L$. Frequencies are defined as before
as occurrences per unit volume in $G_t$. In fact, the ergodic theorem and 
Rohlin towers were first developed for discrete group actions and only later
extended to continuous groups. 

Having $G_r$ nontrivial is more of a complication, especially if $G_r$
is continuous, as with the pinwheel tiling. The ergodic theorem still
applies, since we can first average over $G_r$ and then average
over $G_t$, but the $G$-orbit of a tiling
can no longer be identified with Euclidean $\R^d$.  
When $G_r$ is continuous, the frequency of a patch is no longer ``number per
unit volume'', but is ``number per unit volume per unit angle'', and may 
depend on angle.  

If the group $G'$ that defines our dynamics is the same as the group $G$
used to construct the tiling, then invariant measures are parametrized exactly
as before, by $\poly_\infty$, or equivalently by 
sequences of well-defined supertile frequencies. 
The only difference is that 
$\rho_n(i)$ is the sum or integral over angle of the frequency of the
supertile $P_n(i)$. That is, it counts the average number per unit area 
of $P_n(i)$'s appearing in {\em any} orientation.

If $G'$ is different from $G$, then we must distinguish between the
$G$-invariant measures, which are parametrized by $\Delta_\infty$, and
the $G'$-invariant measures, which may not be. Determining whether
every $G'$-invariant measure is $G_r$-invariant is a separate
computation.

\section{Spectral theory, entropy, and mixing}
\label{Spectral.section}

A vector $\vec \alpha \in \R^d$ is a {\em topological eigenvalue} of
translation if there is a continuous map $f: X_\rrr \to S^1$, where
$S^1$ is the unit circle in $\C$, such that, for every $\T \in X_\rrr$
and every $\vec v \in \R^d$,
\begin{equation}\label{eig-def}
f(\T-\vec v) = \exp(2 \pi i {\vec\alpha} \cdot \vec v) f(\T).
\end{equation} 
The map $f$ is called a {\em topological eigenfunction}.
{\em Measurable} eigenvalues and eigenfunctions 
are defined similarly, only for $f$
measurable rather than continuous. Of course, since
continuous functions are measurable, every topological eigenvalue is a
measurable eigenvalue.  Given a translation-invariant measure $\mu$ on
$X_\rrr$, one can also ask how translations act on $L^2(X_\rrr,
\mu)$. The pure point part of the spectral measure of the translation
operators is closely related to the set of measurable eigenvalues.
$(X_\rrr,\mu)$ is said to have ``pure point spectrum'' if the span of
the eigenfunctions is dense in $L^2(X_\rrr,\mu)$.

The set of topological eigenvalues, the set of measurable eigenvalues, and
the spectral measure of the translation operators are all $G_r$-invariant. 
If $g \in G_r$ and $f$ is an eigenfunction with eigenvalue ${\vec\alpha}$,
then $f \circ g$ is an eigenfunction with eigenvalue $g^{-1}({\vec\alpha})$. 
If $G_r$ is continuous,
then this means that the number of nontrivial eigenvalues is either uncountable
or zero. The first is impossible, as $L^2(X_\rrr,\mu)$ is separable,
and as every topological eigenvalue is a measurable eigenvalue. 
The search for (topological or measurable) eigenvalues only has meaning, then,
when $G_r$ is discrete, in which case we can increase our prototile set by
counting each orientation separately and take $G_r$ trivial.

{\bf Standing assumption for Section \ref{Spectral.section}:}  For the remainder
of this section, we assume $G=\R^d$.  

A measurable dynamical system 
is said to be {\em weakly mixing} 
if there are no nontrivial measurable eigenvalues.
A topological dynamical system is {\em topologically weakly mixing} if there
are no nontrivial topological eigenvalues. For primitive 
substitution tiling spaces
there is no distinction between the two sorts of weak mixing, 
as it has been proven \cite{Host, Clark-Sadun1} that every 
measurable eigenfunction (with respect to the unique invariant measure) can
be represented by a continuous function. 
For fusion tilings, the situation is more subtle. In Theorem \ref{eig-criterion} we develop
necessary and sufficient conditions for a vector to be a topological
eigenvalue.   This theorem is similar to 
Theorem 3.1 of \cite{CGM}, and of earlier 1-dimensional results of \cite{BDM}. 
The key differences are that we work with $\R^d$ actions rather than
$\Z^d$ actions and that we do not assume linear repetitivity.

Unlike the substitution situation it is possible for a fusion tiling space to have a 
measurable eigenvalue that is
not a topological eigenvalue. In example \ref{scrambledFib} we exhibit a fusion tiling space that
has pure point measurable spectrum but that is topologically weakly mixing. After this example was 
announced, it was noted \cite{KS} that the vertices of this tiling form a diffractive point pattern that is not Meyer.  

\subsection{Topological eigenvalues}

For self-affine substitution tilings  there are well-established criteria
for a vector being an (topological or measurable) 
eigenvalue \cite{Sol.self.similar}. 
Given a substitution with 
stretching map $L$,  there is a finite list of vectors
$\vec v_i$ such that ${\vec\alpha}$ is an eigenvalue if and only if 
\begin{equation}
\label{sub-test}
 {\vec\alpha} \cdot L^n (\vec v_i) \to 0 \pmod 1
\end{equation}
for each $i$. 

Our first task is to construct an analogous criterion for 
topological eigenvalues of fusion
tilings. Assuming strong primitivity, each $(n+2)$-supertile contains
multiple copies of each $n$-supertile (at least one per $(n+1)$
supertile).  Let $\vvv^n$ be the set of relative positions of two
$n$-supertiles, of the same type, within an $(n+2)$-supertile.  This
is a finite set, since there are only finitely many kinds of
$(n+2)$-supertiles and each $(n+2)$-supertile contains only finitely
many $n$-supertiles.  We call the elements of $\vvv^n$ {\em return
  vectors}.
For each $\vec\alpha \in \R^d$, let 
$\eta_n({\vec\alpha}) = \max_{\vec v \in \vvv^n} 
|\exp(2 \pi i {\vec\alpha} \cdot \vec v)- 1|.$ 
\begin{thm}\label{eig-criterion}
Let $\rrr$ be a strongly primitive and recognizable van Hove
fusion rule with finite local complexity. A vector
${\vec\alpha} \in \R^d$ is a topological eigenvalue if and only if 
$\sum_n \eta_n({\vec\alpha})$ converges. 
\end{thm}

For primitive substitution tilings,
$\vvv^{n+1} = L \vvv^n$ and $\eta_n({\vec\alpha})$ either goes to zero
exponentially fast, or does not go to zero at all
\cite{Sol.self.similar}. In such cases, the convergence of $\sum_n
\eta_n({\vec\alpha})$ is equivalent to $\eta_n({\vec\alpha})\to 0$, which is
equivalent to the criterion (\ref{sub-test}), where the vectors $\vec
v_i$ range over $\vvv^0$.  

\begin{proof}
  Since $X_\rrr$ is minimal, a continuous eigenfunction with a given
  eigenvalue is determined by its value on a single tiling $\T$.  Fix
  $\T \in X_\rrr$ and ${\vec\alpha} \in \R^d$ and define $f(\T)=1$. For
  each $\vec x \in \R^d$ let $f(\T-\vec x) = \exp(2 \pi i {\vec\alpha} \cdot 
\vec x)$.  If this function is continuous on the orbit of $\T$, then it extends to an eigenfunction
  on all of $X_\rrr$. If it is not continuous, then ${\vec\alpha}$ cannot be
  a topological eigenvalue.


  Suppose that $\sum_n \eta_n({\vec\alpha})$ converges.  We will show that $f$
  is continuous on the orbit of $\T$.  Choose
  $\epsilon > 0$ and pick $n$ large enough that $\sum_{k=n}^\infty
  \eta_n({\vec\alpha}) < \epsilon$. We will show that if $\T-\vec x$ and 
$\T-\vec y$
  agree to the $n$th recognizability radius $\rho_n$, then $f(\T-\vec x)$ and
  $f(\T-\vec y)$ are within $\epsilon$.  The following lemma states that
  $\vec y-\vec x$ can be expressed as a sum of return vectors.

\begin{lem}\label{lemma1}
Suppose that $\vec x$ and $\vec y$ are corresponding points in
$n$-supertiles of the same type within the same $N$-supertile, with
$N \ge n+2$. 
Then $\vec y-\vec x$ can be written as 
$\sum _{k=n}^{N-2} \vec v_k$, where $\vec v_k \in \vvv^k$.
\end{lem}

\begin{proof}[Proof of lemma] 
  For each $n$, we work by induction on $N$. The base case $N=n+2$
  follows from the definition of $\vvv^n$. Now suppose the lemma is
  true for $N=N_0$, and suppose that $\vec x$ and $\vec y$ sit in the same
  $(N_0+1)$-supertile. The point $\vec x$ sits in an $(N_0-1)$ supertile
  $S_x$, say of type $i$, and $\vec y$ sits in an $N_0$-supertile, say of
  type $j$. By strong primitivity, there is an $(N_0-1)$ supertile
  $S_y$ of type $i$ in the $N_0$-supertile that contains $\vec y$. Let $\vec z$
  be the point in $S_y$ corresponding to where $\vec x$ sits in $S_x$.
  (See Figure \ref{evalu-lemma-1}.)
\begin{figure}[h]
\includegraphics[width=6in]{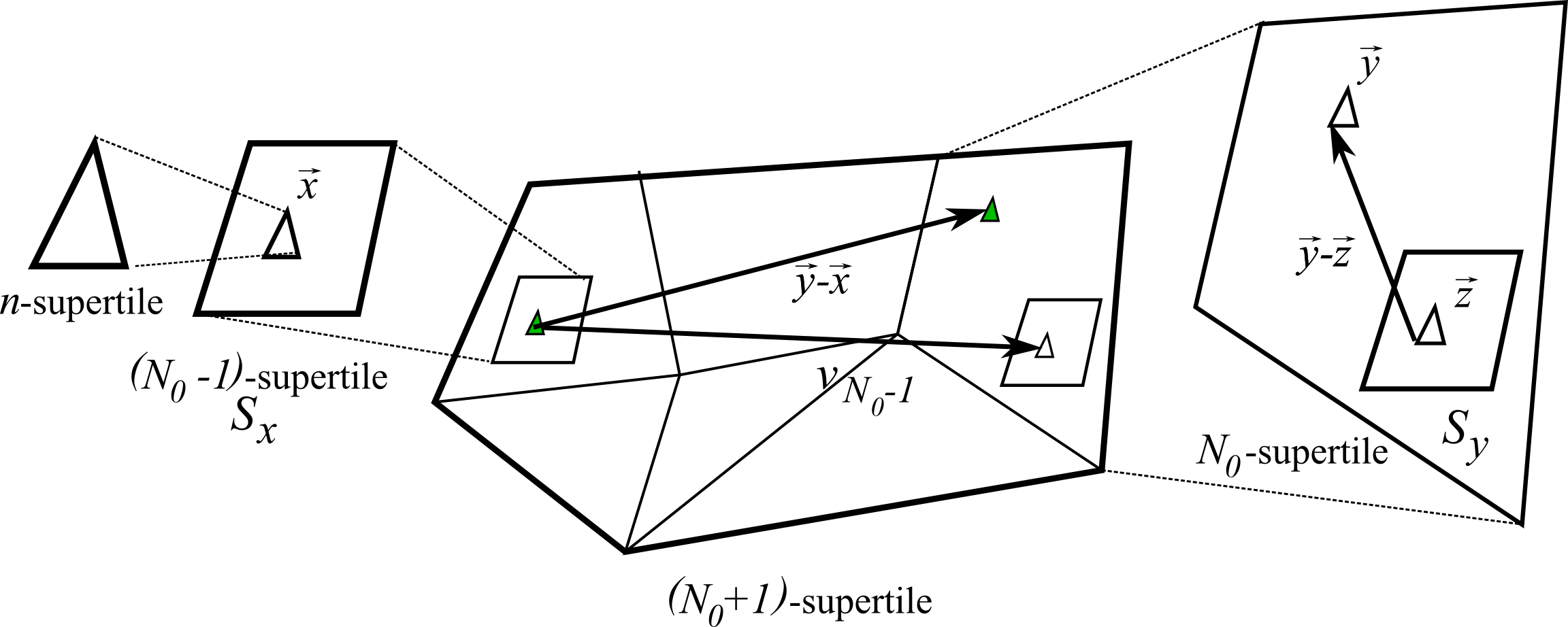}
\caption{The induction step identifies a return between the
  $n$-supertiles (shown as shaded triangles) inside their $(N_0 +
  1)$-supertile.}
\label{evalu-lemma-1}
\end{figure}
Then $\vec z-\vec x$ is a return vector 
from $S_x$ to $S_y$ which we denote $\vec v_{N_0-1}
\in \vvv^{N_0-1}$.
Meanwhile, $\vec y$ and $\vec z$ sit in the same 
kind of $n$-supertile within the
same $N_0$-supertile, so $\vec y-\vec z = 
\sum_{k=n}^{N_0-2} \vec v_k$.   This means that
$\vec y-\vec x = \sum_{k=1}^{N_0-1} \vec v_k$, as desired. 
\end{proof}
 
If $\vec x$ and $\vec y$ lie in the same 
$N$-th order supertile, the lemma implies that 
$\vec y-\vec x = \sum_{k=n}^{N-2} \vec v_k$, so
\begin{eqnarray*} 
|f(\T-\vec y)-f(\T-\vec x)| &=& \left |\exp\left (2 \pi i {\vec\alpha} \cdot 
(\vec y-\vec x)\right )
- 1 \right | \\
\le \sum_{k=n}^{N-2} |\exp(2 \pi i {\vec\alpha} \cdot \vec v_k) - 1| 
& \le & \sum_{k=n}^{N-2} \eta_n({\vec\alpha}) < \epsilon.
\end{eqnarray*}
Even if $\vec x$ and $\vec y$ do not lie in the same $N$-supertile of $T$ for
any $N$, it is still true that any patch containing $\vec x$ and $\vec y$ is
congruent to a patch that lies within an $N$-supertile, so $\vec y-\vec x$ 
still
takes the form $\sum_{k=n}^{N-2} \vec v_k$ for some $N$ and we still
obtain that $|f(\T-\vec y)-f(\T-\vec x)| < \epsilon $. This estimate proves that
$f$ is continuous on the orbit of $\T$. By minimality, it extends to a
continuous eigenfunction on all of $X_\rrr$.  This proves half of
Theorem \ref{eig-criterion}.

For the converse, suppose that $\sum_n \eta_n({\vec\alpha})$ diverges. Then
there exists a subsequence $\sum_k \eta_{n+3k}({\vec\alpha})$ that also
diverges.   We have the following lemma, that states that any finite sum of 
return vectors separated by three levels appears in $X_\rrr$ as the return of
some $n$-supertile to itself.

\begin{lem}\label{lem3} For a given $n$, pick $N$ such that $N+1-n$ is divisible by 3.
For $k = n, n+3, \ldots, 
N-2$ pick $\vec v_k \in \vvv^k$ and let 
$\vec v = \vec v_n + \vec v_{n+3} + \cdots + \vec v_{N-2}$. 
For every such set of choices, there exists 
an  $N$-supertile containing two $n$-supertiles of the same type, 
such that the relative position of the two $n$-supertiles is $\vec v$. 
\end{lem}

\begin{proof}
  Again we work by induction on $N$. If $N=2$, then this follows from
  the definition of $\vvv^n$. Now suppose it is true for $N=N_0$, and
  we shall attempt to prove it for $N=N_0+3$. By the inductive
  hypothesis, there exist points $\vec x_0$ and $\vec y_0$ in corresponding
  $n$-supertiles within the same $N_0$-supertile $S_1$ such that
  $\vec y_0-\vec x_0 = \vec v = \vec v_n + \vec v_{n+3} + \cdots + \vec
  v_{N_0-2}$, and there exist two $(N_0+1)$-supertiles $S_2$ and
  $S_3$, of the same type and with relative position $\vec v_{N_0+1}$,
  within an $(N_0+3)$ supertile.  (See Figure \ref{evalu-lemma-2}).
\begin{figure}[ht]
\includegraphics[width=5in]{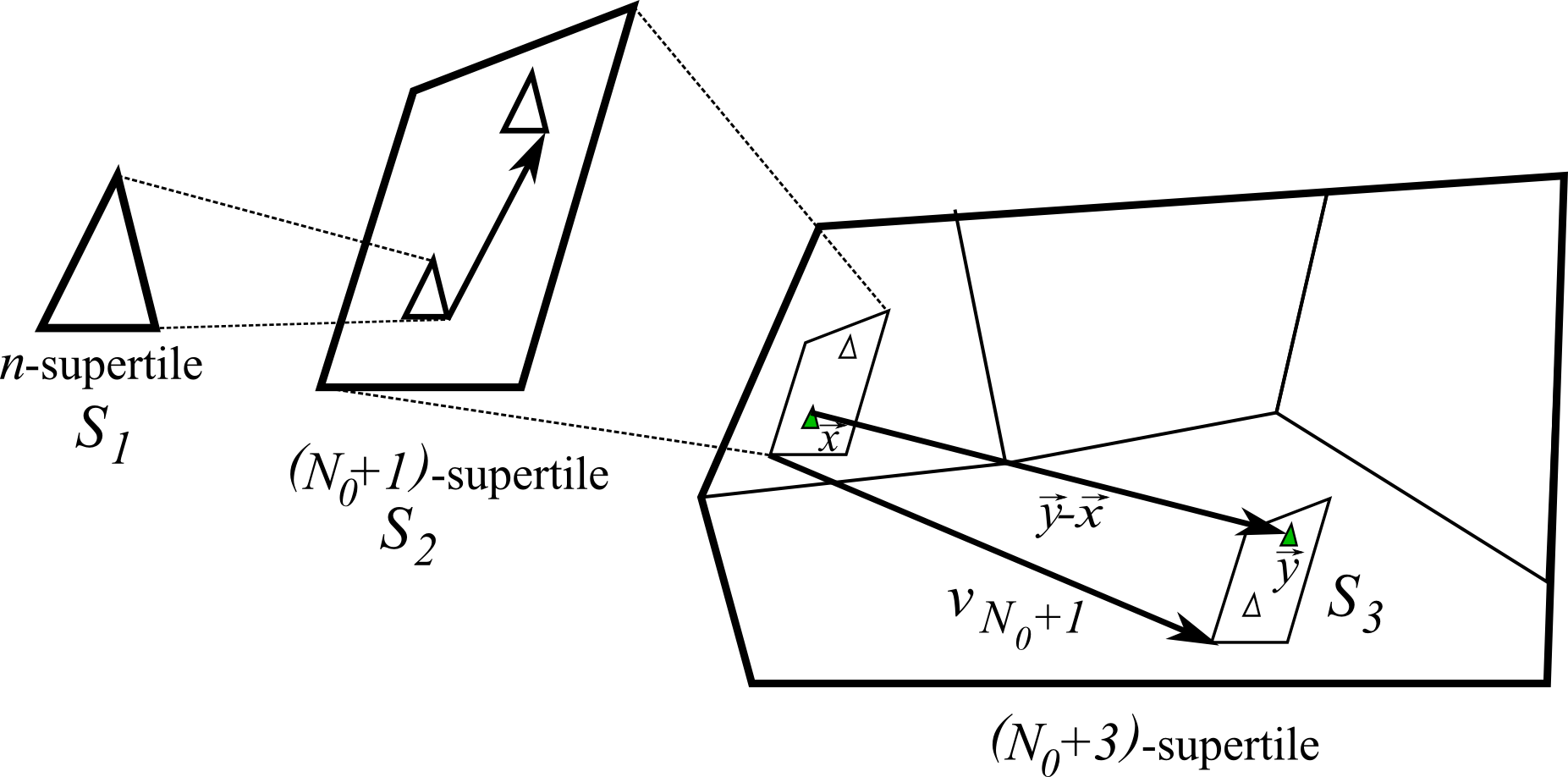}
\caption{The induction step gives a return vector $\vec y-\vec x$ between two
  copies of $S_1$, shown shaded inside an $(N_0 + 3)$-supertile.}
\label{evalu-lemma-2}
\end{figure}
By primitivity, $S_2$ and $S_3$ each contain
copies of $S_1$ (in corresponding positions). 
Let $\vec x$ be the point corresponding
to $\vec x_0$ in the copy of $S_1$ inside $S_2$, and let $\vec y$ be the point 
corresponding to $\vec y_0$ in the copy of $S_1$ inside $S_3$. Then
$\vec y-\vec x = \vec v$.  
\end{proof}

Thus for any $\epsilon$ and for infinitely many values of $n$, one can
find vectors $\vec v_n, \vec v_{n+3}, \cdots, \vec v_{N-2}$ with $\vec
v_k \in \vvv^k$, such that $|\exp(2 \pi i {\vec\alpha} \cdot \sum \vec v_k)
- 1| > 2 \epsilon$.  By restricting to a subsequence we can assume that
the complex numbers $\exp(2 \pi i {\vec\alpha} \cdot \vec v_k)$ are either
all in the first quadrant or all in the fourth quadrant, and that 
$|\exp(2 \pi i {\vec\alpha} \cdot \sum \vec v_k)
- 1| > \epsilon$.  
By Lemma \ref{lem3} there then exist, for $n$
arbitrarily large, two $n$-supertiles of the same type with relative
position $\sum \vec v_k$.  Pick $\vec x$ and $\vec y$ to be corresponding points
of these supertiles in $\T$, such that a big ball around $\vec x$ and $\vec y$
lies entirely in the supertile. (This is possible since the supertiles
form a van Hove sequence.) Then $f(\T-\vec x)$ and $f(\T-\vec y)$ differ
in phase by at least $\epsilon$, 
so our purported eigenfunction is not continuous.
\end{proof}

\subsection{Measurable eigenvalues}

In this section 
we provide an example that has pure discrete spectrum from a measurable
standpoint while being weakly mixing from a topological one.

\begin{ex}{\em The scrambled Fibonacci tiling.}\label{scrambled.fib.1d}
\label{scrambledFib}
We consider four fusions, denoted by the letters $\fff$, 
$\aaa$,  $\eee$ and $\sss$,
with the last being the ``scrambled Fibonacci" fusion whose tiling space has the desired properties.
All use the prototile set $\{a,b\}$ where the length of $a$ is
the golden mean $\phi$ and the length of $b$ is $1$.  

The first fusion
rule is the usual Fibonacci rule $\fff$, which is prototile- and
transition-regular with $(n+1)$-supertiles given by $F_{n+1}(a) =
F_n(a) F_n(b), F_{n+1}(b) =F_n(a)$.  This fusion rule generates the
self-similar Fibonacci tiling space $X_\fff$, which is known to have
measurable and topological
pure point spectrum with eigenvalue set $\frac{1}{\sqrt{5}}\Z[\phi]$.  
Importantly for
our calculations, the Euclidean length of $F_{n-1}(a)$ and $F_{n}(b)$ is
$\phi^{n}$, which deviates from an integer by 
$\pm\phi^{-n}$, and differs from an integer multiple of $1/\phi$ 
by $\pm\phi^{-(n+2)}$.
The transition matrix is $M_0 = \left ( \begin{smallmatrix}
      1&1\\1&0
\end{smallmatrix} \right )$.
%
%
To make the second fusion rule, called ``accelerated''
Fibonacci, we first fix some increasing sequence of positive integers
$\{N(n)\}_{n =1}^\infty$, where we assume
$N(n) - N(n-1) > 2$.  We define $\aaa$ to be the induced fusion rule
on $N(n)$ levels so that $A_n(a) = F_{N(n)}(a)$ and $A_n(b) =
F_{N(n)}(b)$.  The lengths of the $a$ and $b$ $n$-supertiles for
$\aaa$ are $\phi^{N(n)+1}$ and $\phi^{N(n)}$, respectively.  This
fusion rule is prototile-regular but not necessarily transition-regular, since now
the transition matrices $M_n$ are given by $M_0^{N(n) - N(n-1)}$.

The third fusion, which we call ``exceptional'', is
derived from the accelerated rule by introducing a third supertile
type on all odd levels and using it to introduce a relatively small
defect in the next (even) level.  On both the odd and even levels of
$\eee$, the supertiles $E_n(a)$ and $E_n(b)$ are constructed with the
same populations of prototiles and $(n-1)$ supertiles as $A_n(a)$
and $A_n(b)$.  When $n$ is odd, and only when $n$ is odd, there is an
additional tile $E_{n}(e)$ with the same population of $(n-1)$-supertiles
as $E_n(b)$. The
fusion rules for $\eee_n$ are as follows:

When $n$ is odd, $E_n(a)$ and $E_n(b)$ are built from $\eee_{n-1}$ in
exactly the same way that the $\aaa_n$-supertiles are built from
$\aaa_{n-1}$.  The exceptional supertile $E_n(e)$ is obtained from
$E_n(b)$ by permuting the $(n-1)$-supertiles so that all of the
$E_{n-1}(a)$-supertiles come before any of the $E_{n-1}(b)$-supertiles.  
When $n$ is
even, $E_n(a)$ and $E_n(b)$ are built from $\eee_{n-1}$ in exactly the
same way that the $\aaa_n$-supertiles are built from $\aaa_{n-1}$,
except that one copy of $E_{n-1}(b)$ is replaced with $E_{n-1}(e)$.

We can make the fusion rule prototile-regular by taking the induced
fusion of $\eee$ on even levels.  We call this last fusion rule the
{\em scrambled Fibonacci} $\sss$, but most of our proofs center on the
equivalent space $X_{\eee}$.
By controlling the sequence $N(n)$ we can change spectral properties
of the scrambled Fibonacci fusion.  

The space $X_\fff$ is well known to be recognizable, and the recognizability
of $X_\aaa$ is similar. The same patterns that allow us to recognize supertiles
in $X_\fff$ also work (with small modifications) in $X_\eee$ and $X_\sss$. We 
can thus freely speak of the (unique) $n$-supertile containing a 
particular tile. 

\begin{prop}\label{fib.even} 
If $N(2n) - N(2n-1)$ goes to infinity fast enough that 
$\sum_n \phi^{N(2n-1)-N(2n)}$ converges, 
then all four fusion spaces have pure point measurable spectrum with
eigenvalues $\frac{1}{\sqrt{5}}(\Z + \phi \Z)=\frac{1}{\sqrt{5}}\Z[\phi]$.
\end{prop}

\begin{proof} 
We will show that the four spaces $X_\fff$, $X_\aaa$, $X_\eee$ and $X_\sss$ are 
all measurably conjugate.  Then, since $X_\fff$ is well known to have
pure point spectrum with eigenvalue set $\frac{1}{\sqrt{5}}\Z[\phi]$, 
the others must as well.
Since $X_\fff$ and $X_\aaa$ are manifestly the same, and since $X_\eee=X_\sss$,
we need only show that $X_\aaa$ and $X_\eee$ are measurably conjugate. 

In the tilings in $X_{\eee}$, we call a supertile of any level {\em
  exceptional} if it lies in an $E_m(e)$-supertile of some level. Note
that $E_{2m+2}(a)$ and $E_{2m+2}(b)$ each contain only one
$E_{2m+1}(e)$ supertile and a large number (of order $\phi^{N(2m+2)-N(2m+1)}$)
of supertiles of type $E_{2m+1}(a)$ and $E_{2m+1}(b)$. The fraction of
$(2n)$-supertiles that are exceptional in $X_\eee^{2n}$ is thus
bounded by a constant times $\epsilon_n=\sum_{m=n}^\infty \phi^{-(N(2m+2)
  -N(2m+1))}$, which by assumption goes to zero as $n \to \infty$.  

Suppose $T$ is a tiling in $X_\eee$. If the origin lies in an
unexceptional supertile of some level $n$, and hence also at levels
$n+1$, $n+2$, etc., and if the union of these supertiles is the entire
line,\footnote{Note that this condition is translation-invariant, as
  every point in $T$ would then lie in a sequence of unexceptional
  supertiles whose union is the entire line.}  then we can convert
this to a tiling in $X_\aaa$ by replacing each unexceptional
$\eee$-supertile containing the origin with the corresponding
$\aaa$-supertile.  From the definitions of the supertiles, this
operation on $n+1$-supertiles is consistent with the operation on
$n$-supertiles.
 
The measure of the tilings for which the origin is in an exceptional
$n$-supertile is bounded by a constant times $\epsilon_n$, and so 
goes to zero as $n\to\infty$. Thus,
with probability 1, the origin lies in an unexceptional supertile of
some level. Likewise, with probability 1, the union of the supertiles
containing the origin is all of $\R$. Thus we have a map from $X_\eee$
to $X_\aaa$ that is defined except on a set of measure zero.  This map
is readily seen to preserve measure and to commute with translation.
\end{proof}

\begin{prop} \label{fib.odd}
If $\lim_{n \to \infty}N(2n+1) - 2N(2n) = +\infty$, then
  $X_{\sss}$ is topologically weakly mixing.
\end{prop}

\begin{proof}First we show that elements of $\frac{1}{\sqrt{5}}\Z[\phi]$
cannot be topological eigenvalues. 
Then we show that real numbers that are not of this
form cannot be topological eigenvalues.

The supertiles $S_n(a)=E_{2n}(a)$ and $S_n(b)=E_{2n}(b)$ have length
$\phi^{N(2n)+1}$ and $\phi^{N(2n)}$, respectively.
For $\alpha \in \frac{1}{\sqrt{5}}\Z[\phi]$,
$|\exp(2 \pi i \alpha |S_n(a)|)-1|$ and
$|\exp(2 \pi i \alpha |S_n(b)|)-1|$ 
are bounded above and below
by constants (depending on $\alpha$) times $\phi^{-N(2n)}$.
Each supertile $E_{2n+1}(e)$ contains the Fibonacci number 
$f_{N(2n+1) - N(2n)}$ consecutive
copies of $E_{2n}(b)$, since that is how many
$N(2n)$-supertiles of type $b$ there are in the $N(2n+1)$-supertile
$F_{N(2n+1)}(b)$.  We thus find at least that many consecutive copies of
$E_{2n}(b)$ in $S_{n+2}(a)$ and $S_{n+2}(b)$, so there exist vectors
$v_k = k |S_n(b)| $ between $n$-supertiles of the same type in the
same $(n+2)$-supertile, where $k$ is any positive integer up to 
$f_{N(2n+1) - N(2n)}$. Since $|\exp(2 \pi i \alpha |S_n(b)|)-1|$ is 
bounded below by a constant times $\phi^{-N(2n)}$, and since 
$f_{N(2n+1)-N(2n)}$ is of order $\phi^{N(2n+1)-N(2n)}$, and  
since $\phi^{N(2n+1)-2N(2n)}$ grows without bound, 
there are $k$ for which for which
$\exp(2 \pi i \alpha v_k)$ is not close to 1. In fact, by taking $n$
sufficiently large and picking $k$ appropriately, 
we can get $\exp(2\pi i \alpha v_k)$  to be as close as we want to any
number on the unit circle.

On the other hand, if $\alpha$ is not in $\frac{1}{\sqrt{5}}\Z[\phi]$,
then by Pisot's theorem, $\exp(2 \pi i \phi^n \alpha)$ does not
approach 1 as $n \to \infty$. Since for arbitrarily large patches $P$
there exist return vectors of length $\phi^n$ for $n$ sufficiently
large, $\alpha$ cannot be a topological eigenvalue.
\end{proof}
\end{ex}

It is simple to construct sequences $N(n)$ that meet the conditions of 
both Propositions (\ref{fib.even}) and (\ref{fib.odd}). For instance, we
could take $N(n)=3^n$. Thus there exist fusion tilings that are 
topologically weakly mixing but are measurably pure point. 

\subsection{Pure point spectrum}
An important and widely studied problem in substitution
sequences and substitution tilings is determining when the
(measure-theoretic) tiling dynamics have pure point spectrum.  A key
tool is Dekking's coincidence criterion \cite{Dekking}, first developed for
1-dimensional substitutions of constant length and later extended to arbitrary
substitutions, with generalizations in higher dimensions such as Solomyak's overlap algorithm
\cite{Sol.self.similar}. 
In this section we explore the extent to which the
analog of Dekking's criterion determines spectral type for fusions.
The differences between substitutions and fusions are already apparent in the
simplest category, namely one dimensional fusions of constant length. 

We say a 1-dimensional fusion (or substitution) has 
{\em constant length} if, for each $n$,
all of the n-supertiles $P_n(i)$ have the same size. This implies
that tiles all have the same length
and that, for fixed $n$, 
each $n$-supertile contains the same number $L_n$ of $(n-1)$-supertiles.
The fusion is
{\em coincident} if, for each $n$, there exists an $N$ such that any
two $N$-supertiles agree on at least one $n$-supertile.  The fusion is
{\em coincident with finite waiting} if there exists a fixed integer
$k$ such that $N=n+k$ works for every $n$. For substitution tilings,
coincidence is equivalent to coincidence with finite waiting, but for
fusions it is not.  

To each fusion of constant length we associate a solenoid $S_\rrr$,
obtained as the inverse limit of the circle $\R/\Z$ under a series of
maps, with the $n$-th map being multiplication by $L_n$. $S_\rrr$ is a
topological factor of $X_\rrr$, with a point in $S_\rrr$
describing where the origin lies in a tile, a 1-supertile, a
2-supertile, etc, but not generally determining which type of
$n$-supertile the origin sits in.  There is a natural translational
action on $S_\rrr$, and the span of the eigenfunctions of this action
is dense in $L^2(S_\rrr)$. If the factor map from $X_\rrr$ to
$S_\rrr$ is a measurable conjugacy, or equivalently if there is 
a set of full measure on $X_\rrr$ where the factor map is 1:1, 
then $X_\rrr$ has pure point
spectrum. If the factor map is not a conjugacy, and if every
eigenfunction on $X_\rrr$ is obtained from an eigenfunction on
$S_\rrr$,\footnote{This is connected to the {\em height} of a
  substitution or fusion. If a substitution has height one, then all
  eigenvalues of $X_\rrr$ are eigenvalues of $S_\rrr$ 
\cite{Queffelec}.
One can similarly define a notion of height for
  fusions.}  then $X_\rrr$ does not have pure point spectrum.  
  
For substitutions of constant length, the situation is clear-cut: 

\begin{thm}[\cite{Dekking}]
A one dimensional tiling space obtained from a primitive and recognizable
  substitution of constant length and height one has pure point
  measurable spectrum if and only if it is coincident.  
\end{thm}

There are two reasons why this theorem does not apply to general fusions. 
First, a coincident fusion may not be uniquely ergodic. For each ergodic
measure, the question isn't whether a generic point in $S_\rrr$
corresponds to a single tiling, but whether it
corresponds to a single tiling {\em in a suitably chosen set of full measure}. 
Second, a coincidence, or even a coincidence with
finite waiting, only implies that supertiles agree somewhere. Unless
we have some control over the transition matrices, we cannot conclude
that high-order supertiles agree on a fraction approaching 1 of their
length, which is what is needed to obtain a measurable conjugacy between
$X_\rrr$ and $S_\rrr$.

In Example \ref{not-uniquely-ergodic-ex}, 
the fusion is not coincident,
as $P_n(a)$ and $P_n(b)$ disagree at every
site. The map from $X_\rrr$ to $S_\rrr$
is nowhere 1:1, being 4:1 over the orbit 
where there exist two infinite-order supertiles, and 
2:1 over all other orbits. 

However, for each ergodic measure, $X_\rrr$ {\em does} have pure
point spectrum. For instance, for the ergodic
measure that comes from the supertile sequence $\{a,a,\ldots\}$, the
measure of the tilings where the origin sits in a supertile $P_n(a)$
is exponentially close to 1.  With probability 1, for all sufficiently
large $n$ the
$n$-th order supertile containing the origin is of type $a$. 
Also with probability 1, the infinite-order
supertile containing the origin covers the entire line. 
The set of tilings with both these properties has full measure, and 
the factor map to $S_\rrr$ is 1:1 on this set.

\begin{ex} To see how coincidence with finite waiting is insufficient
to prove pure point spectrum we make a fusion based on the substitution   $\sigma(b)=b c^5 b^4$, $\sigma(c)=cb^5c^4$.
Repeated substitution produces words $\sigma^n(b)$ and $\sigma^n(c)$;
by abusing notation we write $\sigma^n(P_{n-1}(b))$ and $\sigma^n(P_{n-1}(c))$ to mean
the fusion of $(n-1)$-supertiles of types $b$ and $c$ in the order given by the letters of
 $\sigma^n(b)$ and $\sigma^n(c)$.   We introduce coincidence with finite waiting by
 defining the fusion rule to be
$$P_n(b)= P_{n-1}(b)
  \sigma^n(P_{n-1}(b)) P_{n-1}(c), \qquad P_n(c)= P_{n-1}(b)
  \sigma^n(P_{n-1}(c)) P_{n-1}(c).$$

The transition matrix
  $M_{n-1,n}=\left ( \begin{smallmatrix} 5&5 \cr 5&5
\end{smallmatrix} \right )^n+\left ( \begin{smallmatrix} 1&1 \cr 1&1
\end{smallmatrix} \right )=(5 \times 10^{n-1}+1)\left ( \begin{smallmatrix} 1&1 \cr 1&1
\end{smallmatrix} \right )$ has rank 1, so the system is uniquely
ergodic.  The length of an $n$-supertile is $\prod_{j=1}^n (10^j +2)$, and $P_n(b)$
and $P_n(c)$ differ on $10^j$ $(n-1)$-supertiles, implying that they differ on $\prod_{j = 1}^n 10^j$ tiles.  Thus $P_n(b)$ and
$P_n(c)$ agree on a positive fraction of their tiles, namely $1 -
\prod_{j=1}^n \frac{10^j}{10^j+2}$.  As $n\to \infty$, this fraction
increases but does not approach 1. 
This implies that one can find disjoint measurable sets of positive measure that map to the same set on the solenoid. Any function
that distinguishes between these sets cannot be approximated by a function on the solenoid, so the span of the eigenfunctions is
not dense and the spectrum is not pure point.
For an example of such a function, let $f(\T)$ equal 1 if the origin is in a $b$ tile and 0 if the origin is in a $c$ tile. 
\end{ex}

\begin{ex}
It is also possible for different ergodic
measures for the same fusion to have different spectral types.
Consider the 1-dimensional non-primitive substitution 
\begin{equation*}
\sigma(a)=a^{10}, \qquad \sigma(b)=bc^5b^4, \qquad \sigma(c)=cb^5c^4
\end{equation*}
of constant length 10. Next consider a 1-dimensional fusion tilings
with three prototiles $a, b, c$, each of unit length.  Using the same abuse of
notation as in the previous example we define the fusion rule
as
\begin{eqnarray*}
P_n(a) &=& P_{n-1}(a) \sigma^n(P_{n-1}(a)) P_{n-1}(b) P_{n-1}(c) \\
P_n(b) &=& P_{n-1}(a) \sigma^n(P_{n-1}(b)) P_{n-1}(b) P_{n-1}(c), \\
P_n(c) &=& P_{n-1}(a) \sigma^n(P_{n-1}(c)) P_{n-1}(b) P_{n-1}(c), 
\end{eqnarray*}
This fusion is coincident with waiting time 1, in that all
$n$-supertiles begin with $P_{n-1}(a)$ and end with $P_{n-1}(b)
P_{n-1}(c)$.  However, that is only 3 out of $L_n=10^n+3$ $(n-1)$-supertiles,
and the $n$-supertiles disagree on the rest! For
large $n$, the supertiles $P_n(a)$, $P_n(b)$ and $P_n(c)$ disagree at
roughly 70\% of their tiles, 97\% of their 1-supertiles, 99.7\% of their
2-supertiles, etc.

The transition matrices
$$ M_{n-1,n} = \begin{pmatrix} 10^n+1 &1&1 \cr 1& 5\cdot 10^{n-1}+1 & 5 
\cdot 10^{n-1}+1 \cr  1& 5\cdot 10^{n-1}+1 & 5 \cdot 10^{n-1}+1 
\end{pmatrix}$$ 
have rank
2. There are two ergodic measures, one coming from the supertile
sequence $\{a,a,\ldots\}$ and the other coming from an arbitrary
sequence of $b$'s and $c$'s.

When we
take the ergodic measure from the sequence $\{a,a,\ldots\}$,
$X_\rrr$ is measurably conjugate to the solenoid $S_\rrr$ 
and has pure
point spectrum. When we take the other ergodic measure, however, there
is a set of full measure where, for all sufficiently large $n$, the
origin is either in $P_n(b)$ or $P_n(c)$, but is not in the two
right-most $n-1$ supertiles within $P_n(b \hbox{ or }c)$.  This set
admits a measure-preserving involution where, for all sufficiently
large $n$, the supertiles $P_n(b)$ containing the origin are replaced
by $P_n(c)$ and vice-versa.  On any set of full measure, the factor map is 
(almost everywhere) at least 2:1, and the tiling dynamical
system does not have pure point spectrum.

In other words, almost every point of the solenoid corresponds to three
tilings. One set of preimages has full measure with respect to the $\{a,a,\ldots\}$
ergodic measure, while the other two preimage sets have full measure with respect 
to the other ergodic measure. Since the first ergodic measure only ``sees''
one preimage, it has pure point spectrum. Since the other measure
``sees'' two preimages, it does not have pure point spectrum.
\end{ex}

To get pure point spectrum from coincidence, we must control the
transition matrices.
\begin{thm}\label{constant-length}
Let $\rrr$ be a primitive, 
recognizable, prototile-regular, 1-dimensional fusion of
constant length. 
If the fusion is coincident with finite waiting, and if 
the transition matrices $M_{n-1,n}$ are uniformly bounded, then
$X_\rrr$ is uniquely ergodic and has pure point spectrum.
\end{thm}

\begin{proof}Suppose that there are $J$ species of prototiles, that the
fusion is coincident with waiting $k$, and that $M_{n-1,n}(i,j) \le C$ for
all $n,i,j$. Then $L_n \le CJ$. Any two $nk$-supertiles agree on at least
one $(n-1)k$-supertile, at least one $(n-2)k$-supertile in each remaining
$(n-1)k$-supertile, at least one $(n-3)k$-supertile in each remaining
$(n-2)k$-supertile, etc. This means that any two $nk$-supertiles agree on
at least a fraction $1- \left (\frac{C^kJ^k-1}{C^kJ^k}\right )^n$ of their
tiles, a fraction that approaches 1 as $n \to \infty$. In 
particular, the density of 
tiles (and likewise, of $n$-supertiles for any fixed $n$) 
is asymptotically the same in all $N$-supertiles as $N \to \infty$,\
implying unique ergodicity. 

A point in $S_\rrr$ thus determines all but a set of density zero of
the tiles in the infinite-order supertile containing the origin. The
probability of there being an undetermined tile in any fixed bounded
region is thus zero.  Since the real line is a countable union of
bounded regions, and since the probability of having two
infinite-order supertiles in a single tiling is also zero, almost
every point in $S_\rrr$ corresponds to a tiling with {\em no}
undetermined tiles. Thus the factor map $X_\rrr \to S_\rrr$ is a
measurable conjugacy, and $X_\rrr$ has pure point spectrum.
\end{proof}

Theorem \ref{constant-length}, while modest in scope, is typical of
the theorems that can be proven about fusions that are not of constant
length, or that are not 1-dimensional. Given any concidence-based test
for pure point spectrum in the category of substitution tilings (e.g.,
the balanced pair algorithm or the overlap algorithm), one can
construct an analogous test for fusions. However, a positive result
from such a test will only demonstrate pure point spectrum if one can
also show that the coincidences happen frequently enough.  This
requires estimates both on how long one must wait for a coincidence,
and on how much the system has grown in the process.

\subsection{Entropy}
Standard results in symbolic and tiling substitution dynamics say that
such systems cannot have positive entropy.  The obstruction is that the
transitions from level to level do not contain much `new' information.
This continues to be the case for fusion tilings when one assumes that
both the number and shapes of supertiles remain fairly
well-controlled.  This section contains a simple example of a 
minimal and uniquely ergodic fusion rule with positive entropy
and a
sufficient condition for a fusion space to have zero
entropy.

Configurational entropy is based on counting configurations, and for this
we need $G$ to be discrete. We therefore assume that
$G=\Z^d$, so we are essentially dealing with subshifts.
Let $\#_n$ be the number of configurations that can appear
in a $d$-dimensional cube of side $n$ (this is the complexity
function).  The configurational entropy is
$$ \lim_{n \to \infty} \frac{\log \#_n}{n^d} $$
For subshifts, configurational entropy is known to be the same
as topological entropy.

Positive entropy implies that there is a lot of randomness in the
system, while unique ergodicity means that all patterns appear with
well-defined frequencies. These ideas might seem to be in conflict,
but Jewett \cite{Jewett} and Krieger \cite{Krieger}
showed that uniquely ergodic
dynamical systems can exhibit a very wide range of dynamical properties, and
in particular can have positive topological entropy. 
The following example is in the spirit of their construction. 

\begin{ex} {\em A strictly ergodic 
fusion rule with positive entropy.}
  We construct a fusion rule $\rrr$ with $\ppp_0=\{a, b\}$
  recursively.  Let $\ppp_1$ be all words of length 3 in which each
  letter appears at least 1 time but no more than 2 times; we have
  $j_1=6$ distinct 1-supertiles. Now let $\ppp_2$ be all fusions of
  $\frac{3j_1^2}{2}=54$ 1-supertiles in which each supertile appears between
  $j_1$ and $2j_1$ times.  The expected number of 1-supertiles in any
  fusion of $\frac{3j_1^2}{2}$ of them is $\frac{3 j_1}{2}$, 
so we are including the
  highest-probability fusions in our set $\ppp_2$.

  In general, let $j_n$ be the number of $n$-supertiles and let
  $\ppp_{n+1}$ be all fusions of $\frac{3 j_{n}^2}{2}$ $n$-supertiles in which
  each $n$-supertile appears between $j_n$ and $2j_n$ times. Since
  having more than $2j_n$ or fewer than $j_n$ occurrences in a span of
  size $\frac{3j_n^2}{2}$ is already highly improbable, restricting to these
  $(n+1)$-supertiles only reduces the number of configurations
  slightly, and the system so constructed has positive entropy.  The
  transition matrices are enormous and grow super-exponentially but
  always have all nonzero entries, making the system strongly
  primitive and hence minimal.  Moreover, the constant $\delta_n$ used
  in equation (\ref{checkable_condition}) (to measure how balanced the
  columns of the transition matrices are) is always $\frac{j_n}{2j_n} =1/2$,
  so the tiling space is uniquely ergodic.
\end{ex}

The fusion rule $\rrr$ is not recognizable, but we can build a recognizable
fusion rule  $\rrr'$ from $\rrr$ as in Example \ref{recognize-ex}.  Since the
  entropy of the factor $X_{\rrr}$ is bounded by the entropy of
  $X_{\rrr'}$, $X_{\rrr'}$ has positive entropy. It is easy to check that
  the addition of subscripts does not affect unique ergodicity.

This example involved the number $j_n$ 
of $n$-supertile types growing exponentially
with the size of the supertiles. If the growth is slower than exponential,
and if the shapes of the supertiles are not too distorted, then
the system will have zero entropy. 

\begin{prop}\label{zero.entropy}
Let $d_n$ be the diameter of the largest $n$-supertile, let $j_n$
be the number of $n$-supertile types and suppose that there exist
constants $\beta, K$ such that each cube of side $\beta d_n$ touches
at most $K$ $n$-supertiles. If $\lim_{n \to \infty} \frac{\log j_n}{d_n^d} = 0$,
then the configurational entropy of $X_\rrr$ is zero. 
\end{prop}

\begin{proof} To determine the configuration in a cube of side
$\beta d_n$, one must specify the kinds of $n$-supertiles that intersect that
cube, and also specify the locations of those supertiles. There are at 
most $j_n^K$ choices for the first, and at most $V^K$ choices for the second,
where $V$ is the volume of the largest $n$-supertile, which is bounded by
$d_n^d$. Thus the log of the
number of configurations, divided by the volume of the cube, is bounded by
$\frac{K \log(j_n) + K d \log d_n}{\beta^d d_n^d}$, which goes to zero as
$n \to \infty$.
\end{proof}

The upshot of Proposition \ref{zero.entropy} is that positive entropy 
either requires the number $j_n$ of $n$-supertiles to grow 
exponentially with volume, 
or for the shapes and relative sizes of supertiles to be so wild, and for 
the ways that supertiles fit together be so varied, that there are many ways
for supertiles to fit together.  The geometric issues do not apply to 
dimension 1, where supertiles simply concatenate, but could in principle apply
in dimensions 2 or more.  However, we know of no examples where positive 
entropy is achieved without $j_n$ growing exponentially fast.

\subsection{Strong mixing}
A (measurable) dynamical system is {\em strongly mixing} if for any pair
of  measurable sets  $A,B$, and for any sequence of vectors $\vec v_n$ 
tending to infinity, 
$\lim \mu(A \cap (B-\vec v_n))= \mu(A)\mu(B)$.  The dynamical systems of
primitive substitution sequences
and self-similar tilings are never strongly mixing \cite{Dekking-Keane,
Sol.self.similar}. 
Because of the rigidity of the substitution process, knowing the location of one
copy of a patch gives a higher probability that it will be seen
again at certain intervals.  However, there are ``staircase''
cut-and-stack transformations in one and several dimensions that have
been shown to be strongly mixing \cite{Adams, Adams-Silva}, thus it is
possible to have strongly mixing fusion tiling systems.  As in the
case of entropy, this is only possible when the system has increasing
complexity at higher levels of the hierarchy. 

In this section we establish sufficient criteria for fusion tilings
not to be strongly mixing. These  criteria involve both uniform
bounds on the number of supertiles and on the transition
matrices, and are not necessary criteria.
For instance, the accelerated Fibonacci fusion discussed in Example
\ref{scrambledFib} does not have bounded matrices, but is essentially
the same as ordinary Fibonacci and is not strongly mixing.

\begin{thm}
  The dynamical system  of a strongly primitive
  van Hove fusion rule with a constant number of supertiles at each
  level and bounded transition matrices, and with group $G=\R^d$, 
cannot be strongly mixing.
\end{thm}

\begin{proof}
Our proof is an adaptation of Solomyak's \cite{Sol.self.similar},
which in turn is an adaptation of Dekking and Keane's
\cite{Dekking-Keane}. By Corollary \ref{primitive.ergodic},
$X_\rrr$ is uniquely ergodic, so for any patch $P$, 
$freq_\mu(P)$ can be computed from the actual frequency of $P$ in
any increasing sequence of supertiles. 
We will find a patch $P$ and a sequence of vectors $\vec v_n$,
tending to infinity, such that the frequency of $P \cup (P+\vec v_n)$
is bounded away from zero.  Then, supposing that $freq_\mu(P)=\delta$ and
$freq_\mu(P \cup (P+\vec v_n)) > \epsilon$, we pick a set $U \subset \R^d$
whose volume is less than $\frac{\epsilon}{2\delta^2}$, 
and which is small enough that
$\mu(X_{P,U})=freq_\mu(P) Vol(U)$.  Let $A=B=X_{P,U}$. 
Since $X_{P \cup (P+\vec v_n), U} \subset A \cap (B-\vec v_n)$, we have 
\begin{equation}
\mu(A \cap (B-\vec v_n)) \ge \mu(X_{P \cup P+\vec v_n, U}) 
\ge \epsilon Vol(U) > 2 \delta^2 Vol(U)^2 = 2 \mu(A)\mu(B),
\end{equation}
so $\mu(A \cap B - \vec v_n)$ cannot approach $\mu(A)\mu(B)$ as $n \to \infty$.

To find the vectors $\vec v$, we suppose the number of supertiles at each 
level is the constant $J$.
Since $X_\rrr^n$ can be expressed as the union of $J$ cylinder sets
  defined by which $n$-supertile is at the origin, it must be that at
  least one of those cylinder sets has measure at least $1/J$.  For
  each $n$, choose $l_n \in \{1, 2, ... J\}$ corresponding to an
  $n$-supertile with this property, so that $Vol(P_n(l_n)) \rho_n(l_n)
  \ge 1/J$, where $\rho_n$ is the supertile frequency vector.
  Now choose $\vecv_n \in \vvv^n$ to be a return vector for $P_n(l_n)$
  inside $P_{n+2}(l_{n+2})$, as in Theorem \ref{eig-criterion}.
  Because of strong primitivity and the fact that our transition
  matrices are uniformly bounded, we can find a $\delta'> 0$ for which
  $\frac{Vol(P_n(l_n))}{Vol(P_{n+2}(l_{n+2}))} \ge \delta' \text{ for
    all } n.$ (Specifically, if each
$n$-supertile contains at most $K$ $(n-1)$-supertiles, then the ratio of
volume between the largest and smallest $n$-supertile is at most $K$, and 
$Vol(P_n(l_n))/Vol(P_{n+2}(l_{n+2})) \ge 1/K^3$.)

The patch $P$ is arbitrary.
By choosing $n$ large we can make $\frac{\#(P \text{ in }
    P_n(l_n))}{Vol(P_n(l_n))}$ arbitrarily close to $freq_\mu(P)$, and hence
greater than a fixed constant $freq_\mu(P)/2$ for all $n$. 
The reader can refer to Figure
  \ref{counting.P.again} to see that $\#(P \cup P+\vecv_n \text{ in } 
P_{n+2}(l_{n+2})) \ge \#(P \text{ in } P_n(l_n))$. 
Since the fraction of volume from the supertiles $P_n(l_n)$ is at least $1/J$,
this implies that the frequency
of $P \cup (P+\vec v_n)$ is at least $\frac{freq_\mu(P) Vol(P_n(l_n))}
{2 J Vol(P_{n+2}(l_{n+2}))}$,
hence at least $\frac{\delta' freq_\mu(P)}{2J}$.
\begin{figure}[h]
\includegraphics[width=4.5in]{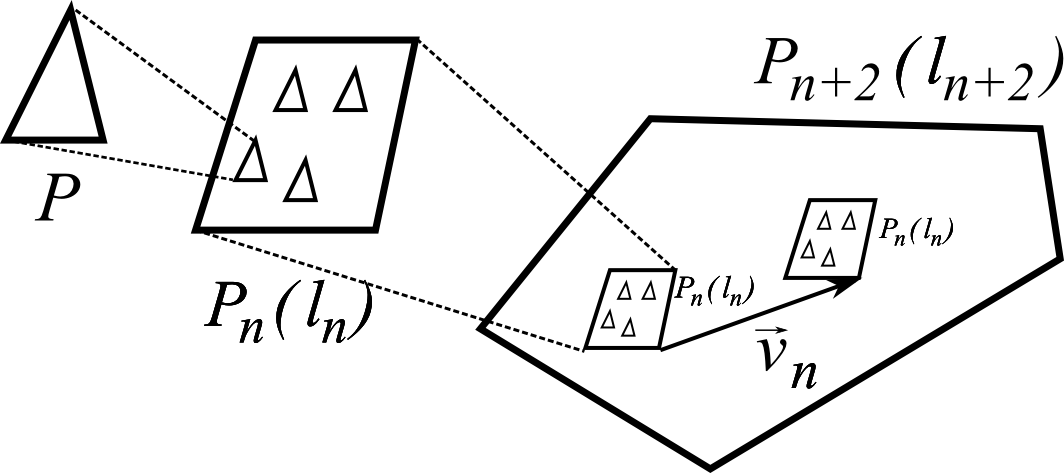}
\caption{Each copy of $P$ in $P_n(l_n)$ makes a copy of $P \cup (P +
  \vecv_n)$ in $P_{n+2}(l_{n+2})$.}
\label{counting.P.again}
\end{figure}
\end{proof}

\begin{prop}
  The dynamical system  of a strongly primitive
  van Hove fusion rule with a constant number of supertiles at each
  level and bounded transition matrices, and with group $G=\Z^d$, 
cannot be strongly mixing.
\end{prop}

\begin{proof}The previous proof does not apply to $\Z^d$ actions because we 
cannot choose $U$ arbitrarily small. However, we have already shown that
for {\em any} patch $P$ and any sufficiently large $n$, there exist large
$\vec v$ with
$freq_\mu(P \cup (P+\vec v)) \ge \delta'
freq_\mu(P)/2J$. We then find a patch $P$ whose frequency is less than 
$\delta'/4J$, 
so that
$freq_\mu((P \cup (P+\vec v_n)) \ge 2 freq_\mu(P)^2$. 
Taking $U$ to consist of one point, 
this implies that $\mu(A \cap (A-\vec v_n)) \ge 2 \mu (A)^2$, where $A=X_{P,U}$. 
\end{proof}

\section{Inverse limit structures, collaring, and cohomology}
\label{Topology.section}

In this section we consider topological properties of spaces $X_\rrr$ of
fusion tilings, including their structure as inverse limit spaces,
their \v Cech cohomology groups, and the significance of these groups. 

{\bf Standing assumptions for Section \ref{Topology.section}:} We assume that $G=\R^d$. Unlike in
Section \ref{Spectral.section}, this is more for convenience than from necessity.  Modifications for other
groups can be done exactly as for substitution tilings \cite{ORS,Tilings.book}. 
We also assume that our
fusion rules are recognizable and, as always, have finite local complexity. 

Tiling spaces can always be represented as inverse limits of CW
complexes  \cite{BBG,Sadun.inverse.limit}.
The challenge is finding a representation that allows for 
efficient calculation and for the proving of theorems.   To this end
we present generalizations of the Anderson-Putnam complex \cite{AP} and of the
partial collaring scheme of Barge, Diamond, and their 
collaborators \cite{BD, BDHS}.  (See also \cite{Gahler-Maloney} for another method of computing
the \v Cech cohomology of transition-regular 1-dimensional fusion tiling spaces that meet additional
assumptions.)

In all cases, we construct a sequence of spaces and maps
$$ \Gamma_0 \xleftarrow{f_0} \Gamma_1 \xleftarrow{f_1} \Gamma_2 \xleftarrow{f_2}
\Gamma_3 \xleftarrow{f_3} \cdots,
$$
where each {\em approximant} $\Gamma_i$ describes a region of the
tiling, each $\Gamma_{i+1}$ describes a larger region of the tiling,
and $f_i: \Gamma_{i+1} \to \Gamma_i$ is the {\em forgetful map} that
loses the additional information carried in $\Gamma_{i+1}$. The {\em
  inverse limit} $\ilim (\Gamma, f)$ is the set of infinite sequences
$(x_0, x_1, \ldots)$ such that each $x_i \in \Gamma_i$ and each
$x_i=f_i(x_{i+1})$. Such a sequence is a set of consistent
instructions for tiling larger and larger regions of $\R^d$. If the
union of these regions is all of $\R^d$ for all sequences in the
inverse limit, then there is a natural homeomorphism between
$\ilim(\Gamma, f)$ and $X_\rrr$.  The key is to make sure that every
tiling in $X_\rrr$ can be built up from the approximants in a unique
way.  A common obstruction is when the approximants can build an
infinite tiling that covers only a portion of $\R^d$.
``Border-forcing'' fusions, discussed next, do not have this
obstruction.  Later we will describe the technique of ``collaring'' to
make fusion rules become border-forcing.

\subsection{Forcing the border}
A fusion rule always tells us how $n$-supertiles make up the interiors
of larger $N$-supertiles.  But sometimes the $N$-supertiles also
determine which $n$-supertiles are on their exterior as well.  When
this happens we say the fusion rule forces the border, and we have a
natural way to see the space as an inverse limit.
\begin{dfn} A fusion rule {\em forces the border} if for each integer
$n$ there exists an $N$ with the following property: If $S_1$ and $S_2$ 
are two $N$-supertiles of the same type appearing in tilings $T_1$ and $T_2$
in $X_\rrr$, then the patch of $n$-supertiles
that touch $S_1$ in $T_1$ is equivalent to the patch of $n$-supertiles touching
$S_2$ in $T_2$.
\end{dfn}

\begin{ex}{\em Compare and contrast: border forcing.} The
  1-dimensional substitution $a\to abb$, $b \to abbb$ forces the
  border in that every $n+1$-supertile of type $a$ is preceded by an
  $n$-supertile of type $b$ and followed by an $n$-supertile of type
  $a$, and likewise every $n+1$-supertile of type $b$ is also prededed
  by an $n$-supertile of type $b$ and followed by an $n$-supertile of
  type $a$.  By contrast, the substitution $a \to ab$, $b \to aa$ does
  not force the border, since an $N$-supertile of type $a$ can be
  preceded either by an $n$ supertile of type $a$ or $b$.
\end{ex}

\subsection{The Anderson-Putnam complex}

To build $\Gamma_0$, we start out with one copy of each prototile from
$\ppp_0$. Then, if somewhere in some tiling two prototiles meet, we
identify the corresponding points on their boundaries. The resulting
branched manifold is compact \cite{AP}.  (If we take the periodic
tiling of unit squares lined up edge-to-edge, it is easy to see that
$\Gamma_0$ is the torus.)  We build $\Gamma_1$ by taking one copy of
each supertile from $\ppp_1$ and identifying the boundaries whenever
they meet as above, and continue making each approximant $\Gamma_n$
similarly.  Put another way, $\Gamma_n$ for the space $X_\rrr$ is
$\Gamma_0$ for $X^n_\rrr$.

There is a natural map from $X_\rrr$ to $\Gamma_n$ that maps a tiling to
the location of the origin within its $n$-supertile. Thus, a point in
$\Gamma_n$ can be viewed as a set of instructions for placing an $n$-supertile
containing the origin.
Or course, if we know the $(n+1)$-supertile containing the origin, then
we necessarily know the $n$-supertile containing the origin, so the 
forgetful map $f_n$ is well-defined. 
 
\begin{thm}\label{border-forcing-homeo}
If the recognizable 
fusion rule $\rrr$ forces the border, then $X_\rrr$ is homeomorphic
to the inverse limit $\ilim (\Gamma_n, f_n)$ of Anderson-Putnam complexes. 
\end{thm}

\begin{proof}We will construct a homeomorphism from the inverse limit
  to $X_\rrr$ by constructing maps from each approximant to partial
  tilings of $\R^d$.  Pick an increasing sequence of integers $n_1,
  n_2, \ldots$ such that all the $n_i$-supertiles bounding an
  $n_{i+1}$-supertile are determined by the type of the
  $n_{i+1}$-supertile. Our map takes a point $x_N$ in $\Gamma_N$ to
  the $N$-supertile with the origin where $x_N$ is, together with all
  of the lower order supertiles that are determined by that
  $N$-supertile. If $N \ge n_i$, then this includes all the
  $n_{i-1}$-supertiles touching the $N$-supertile, all the $n_{i-2}$
  supertiles touching the $n_{i-1}$ supertiles, all the $n_{i-3}$
  supertiles touching the $n_{i-2}$ supertiles, and so on.  If $x_N$
  is on the boundary of an $N$-supertile, then there are multiple
  tilings that can come from this process, but they all agree on the
  $n_{i-1}$-supertiles in all directions around the origin, as well as
  the lower-order supertiles determined by the
  $n_{i-1}$-supertiles. In particular, $x_N$ determines at least $i-1$
  layers of supertiles of various sizes around the origin, and so
  determines at least $i-1$ layers of ordinary tiles around the
  origin.  By choosing $N$ large enough, we can get $i$ to be
  arbitrarily large.  Thus as $N \to \infty$ the points in the
  approximants determine tilings of larger and larger balls around the
  origin, so a point in the inverse limit gives a set of consistent
  directions for tiling increasing regions of $\R^d$ whose union is
  all of $\R^d$. Such instructions are clearly in 1:1 correspondence
  with tilings of $\R^d$.  Checking that the topology of $X_\rrr$
  corresponds to the topology of the inverse limit (as a subset of the
  infinite product $\prod \Gamma_n$) is a straightforward exercise
  that is left to the reader.
\end{proof}

\begin{ex} {\em A short \v Cech cohomology computation.} Consider a
  transition-regular fusion rule in one dimension, with two tile types $a$
  and $b$.  Let $P_n(a)=P_{n-1}(a)P_{n-1}(b) P_{n-1}(b)$ and let
  $P_n(b)$ be a permutation of two $P_{n-1}(a)$'s and three
  $P_{n-1}(b)$'s, with the permutation depending on the level. As long
  as a permutation beginning in $P_{n-1}(a)$ occurs infinitely often
  and a permutation ending in $P_{n-1}(b)$ occurs infinitely often,
  this fusion rule forces the border. The approximant $\Gamma_n$
  consists of two intervals, one representing $P_{n}(a)$ and one
  representing $P_n(b)$, with all four endpoints identified to form a
  figure-8. The map $f_n$ wraps the $P_{n+1}(a)$ circle around the
  $P_n(a)$ circle once and then around the $P_n(b)$ circle twice. It also
  wraps the $P_{n+1}(b)$ circle around the $P_n(a)$ circle twice and
  around the $P_n(b)$ circle three times, in an order determined by the
  fusion rule at level $n+1$.  By Theorem \ref{border-forcing-homeo},
  $X_\rrr$ is the inverse limit of these figure-8's under these maps.

  From a \v Cech cohomology standpoint we can see the figure-8 as the
  chain complex of each approximant, so that both $\check
  H_1(\Gamma_n)$ and $\check H^1(\Gamma_n)$ are isomorphic to $\Z^2$.
  The first \v Cech cohomology of the inverse limit (and of $X_\rrr$)
  is the direct limit of $\Z^2$ under the pullback of the maps $f_n$,
  which always come out to be $\left ( \begin{smallmatrix} 1&2 \cr
      2&3 \end{smallmatrix}\right )$ even though the order for the
  $P_n(b)$'s varies.  Since that matrix is invertible over $\Z$, we
  see that $\check H^1(X_\rrr)=\Z^2$.
\end{ex}

\subsection{Anderson-Putnam collaring}

Most fusion rules, like most substitution rules, do not force the
border.  However there is a simple trick, due to Anderson and Putnam
in the setting of substitutions, for replacing an arbitrary fusion
rule $\rrr_0$ with a hierarchical rule $\rrr$ that forces the border, such
that $X_{\rrr_0}$ is homeomorphic (and topologically conjugate) 
to $X_\rrr$.  We can then express our original tiling space
$X_{\rrr_0}$ as the inverse limit of the Anderson-Putnam complexes of
$\rrr$.

A {\em collared tile to distance $r$}, or an $r$-collared tile, is a
tile together with a label that describes the types and relative
positions of all of that tile's neighbors out to distance $r$.  For
instance, in a 1-dimensional tiling with patch $abbaba$, the three
$b$'s are all different as collared tiles to distance 1, as one is
preceded by an $a$ and followed by a $b$, one is preceded by a $b$ and
followed by an $a$, and one is both preceded and followed by $a$'s.
Likewise, a collared $n$-supertile to distance $r$ is an
$n$-supertile, together with a label indicating the pattern of nearby
$n$-supertiles out to distance $r$.

Take an infinite increasing sequence of radii $r_0 < r_1 < \cdots$,
tending to infinity.  We take $\ppp_n$ 
to be the set of collared $n$-supertiles to distance $r_n$.
Clearly, any complete tiling can be
equally well-described in terms of (super)tiles or collared
(super)tiles.  However, by construction, $\rrr$ forces the border,
since if $r_N$ is greater than $r_n$ plus the diameter of the largest
$n$-supertile, then a collared $N$-supertile determines its
surrounding $r_n$-collared $n$-supertiles.

Note that the label of a collared $n$-supertile contains information about
all the neighboring $n$-supertiles out to distance $r_n$, and in particular
determines all of the neighboring $(n-1)$-supertiles out to distance $r_{n-1}$.
This means that each collared $n$-supertile is uniquely decomposed as a union
of collared $(n-1)$-supertiles, and gives a well-defined map from $X^n_\rrr$ to 
$X^{n-1}_\rrr$. The hierarchical rule $\rrr$ is a generalized fusion in
the sense of Footnote \ref{labeled.supertiles},
since the $(n-1)$-supertiles contained in an $n$-supertile do
not determine the $n$-supertile. The collared $n$-supertiles have strictly
more information than the collared $(n-1)$-supertiles, which is the whole
point of collaring!

If the shapes of the supertiles are not too wild, one can pick
the $r_n$'s to grow slowly compared to the size of the supertiles, so
that collaring to distance $r_n$ just means specifying the nearest
neighbors of the $n$-supertile, as is usually done for substitution
tilings.  However, for some fusion rules it is possible that knowing
the $n$-supertiles containing the origin and the ones touching this
supertile, for all $n$, will not determine the tiling of all of
$\R^d$.

The process of collaring does have its drawbacks, as $\rrr$ may not
have the same transition-regularity or even prototile-regularity
properties as $\rrr_0$.  Collaring increases the number of tile types,
and there is no reason to expect the increase to be the same at all
levels. Indeed, even if the number of uncollared supertiles is
uniformly bounded it is entirely possible that the number of collared
$n$-supertiles will grow without bound as $n \to \infty$. This happens,
for instance, in Example \ref{nonpisot}.

\subsection{Barge-Diamond collaring}

The idea behind Barge-Diamond collaring \cite{BD, BDHS}
is to collar points rather
than tiles. As before, pick an increasing sequence of radii $r_0 <
r_1< \cdots$ tending to infinity. Take a tiling $\T$, and identify
points $\vec x$ and $\vec y$ if $[B_{r_0}]^{\T-\vec x}= [B_{r_0}]^{\T-\vec y}$. 
That is, if
the tiling looks the same around $\vec x$ and $\vec y$ to distance $r_0$ (with
$\vec x$ and $\vec y$ playing corresponding roles).  Let $\Gamma_0$ be the
quotient space. To get $\Gamma_1$, identify points for which the
corresponding tiling in $X^1_\rrr$ agrees to distance $r_1$. That is,
points $\vec x$ and $\vec y$ for which all of the 1-supertiles that exist
within a distance $r_1$ of $\vec x$ and $\vec y$ agree. Likewise, $\Gamma_n$ is
$\R^d$ modulo identification of points $\vec x$ and $\vec y$ for which all of
the $n$-supertiles within distance $r_n$ of $\vec x$ and $\vec y$ agree.

As before, we have a map from $X_\rrr$ to $\Gamma_n$, taking a tiling
to a description of how a ball of radius $r_n$ around the origin sits
in one or more $n$-supertiles. Since $r_n \to \infty$ as $n \to
\infty$, a point in the inverse limit is a consistent set of
instructions for tiling all of $\R^d$, so the inverse limit is
homeomorphic to $X_\rrr$ as long as the orbit closure of $\T$ is $X_\rrr$ (which is always the case when $X_\rrr$ is minimal).

When the fusion is asymptotically self-similar or self-affine, 
one can take $r_n$ to be much smaller than
the size of an $n$-supertile, but still to go to infinity. For
2-dimensional tilings, this means that there are three kinds of
points.  Most points are farther than $r_n$ from the nearest $n$-supertile
boundary. These points are identified with corresponding points of
other $n$-supertiles, without regard for the supertile's neighbors.
Some points are within $r_n$ of one of the supertile's edges. These
points are identified with corresponding points of other
$n$-supertiles that have the same $n$-supertile neighbor across the
specific edge.  Finally, some points are within $r_n$ of two or more
edges, and hence are close to a vertex. 
There is a stratification of $\Gamma$ into
points-near-vertices, points-near-edges, and interior points, and this
stratification makes for much easier computations of tiling cohomology
than Anderson-Putnam collaring.

\section{Direct product variations}
\label{Examples}
An easy way to make higher-dimensional substitution sequences is to
take the direct product of two or more one-dimensional substitutions.
To break the direct product structure, one can rearrange the
substitution carefully so that at each stage the blocks still fit,
creating what is called a {\em direct product variation} or {\em DPV}.
Introduced as examples of combinatorial substitutions in
\cite{My.primer}, DPVs are quite flexible when viewed as examples of
fusion rules.
\begin{ex}{\em The Fibonacci DPV.}\label{FibDPV}
This simple example of a prototile- and transition-regular fusion rule
in two dimensions is based on the Fibonacci substitution $0 \to 01, 1 \to 0$.
We use it to illustrate almost all of the ideas and computations discussed
for fusion rules.

The prototile set consists of four unit-square tiles with label set
$\{a, b, c, d\}$ and so $\ppp_0 =
\left\{\raisebox{-.2cm}{\includegraphics[width=.6cm]{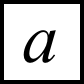}},
\raisebox{-.2cm}{\includegraphics[width=.6cm]{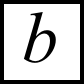}},
\raisebox{-.2cm}{\includegraphics[width=.6cm]{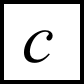}},
\raisebox{-.2cm}{\includegraphics[width=.6cm]{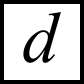}}\right\}$.\footnote{There is some flexibility
with the geometry of the prototiles.  They could be parallelograms or rectangles, and
there are two vertical and two horizontal degrees of freedom for the lengths of the sides.}  

For the $1$-supertiles we choose $\ppp_1 =
\left\{\raisebox{-.4cm}{\includegraphics[width=1.2cm]{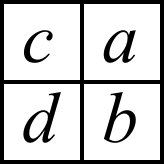}},
\raisebox{-.4cm}{\includegraphics[width=1.2cm]{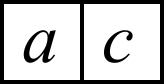}},
\raisebox{-.4cm}{\includegraphics[width=.6cm]{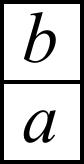}},
\raisebox{-.4cm}{\includegraphics[width=.6cm]{fibprototile-a.png}}
\right\},$ where we list the supertiles in the obvious order
$\{P_1(a), P_1(b), P_1(c), P_1(d)\}$.  To make the $2$-supertiles we
concatenate the $1$-supertiles in combinatorially the same way:
$$ \ppp_2 = \left\{\raisebox{-.8cm}{\includegraphics[width=1.8cm]
{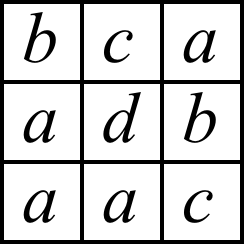}},\raisebox{-.8cm}{\includegraphics[width=1.8cm]
{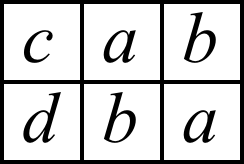}},\raisebox{-.8cm}{\includegraphics[width=1.2cm]
{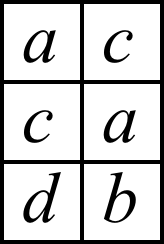}},
  \raisebox{-.8cm}{\includegraphics[width=1.2cm]{fiblevel1-a.png}}\right\}
=
\left\{\raisebox{-.8cm}{\includegraphics[width=2.05cm]{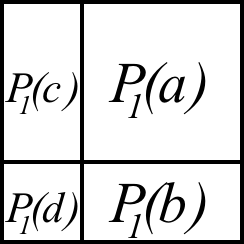}},
\raisebox{-.8cm}{\includegraphics[width=2.05cm]{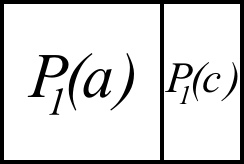}},
\raisebox{-.8cm}{\includegraphics[width=1.2cm]{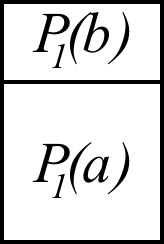}},
\raisebox{-.8cm}{\includegraphics[width=1.2cm]{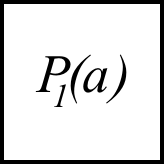}}
\right\}
$$
In general we construct $\ppp_{n+1}$ from $\ppp_n$ with exactly the
same combinatorics as the rightmost version of the $2$-supertiles
shown above.  It is not difficult to show that the sides of $P_n(a)$
and the long sides of $P_n(b)$ and $P_n(c)$ are the Fibonacci numbers
$f_{n+2}$, while the sides of the $n$-supertile of type $d$ and the
short sides of the $b$ and $c$ supertiles are the Fibonacci numbers
$f_{n+1}$ (using the convention that $f_0 = 0$ and $f_1 = 1$).  This
means that at each stage, the supertiles fit together to form squares
and rectangles with Fibonacci side lengths.

Recognizability is straightforward and proceeds by induction. 
The $P_{n+1}(a)$ supertiles are determined by the presence of a $P_n(d)$,
each $P_{n+1}(b)$ is determined by a $P_n(c)$ that is not in a $P_{n+1}(a)$,
each $P_{n+1}(c)$ is determined by a $P_n(b)$ that is not in a $P_{n+1}(a)$,
and each remaining $P_n(a)$ is a $P_{n+1}(d)$.  

The Fibonacci DPV is transition-regular with $M =M_{n-1,n} =
\left(\begin{smallmatrix}
    1&1&1&1\\1&0&1&0\\1&1&0&0\\1&0&0&0 \end{smallmatrix} \right)$. This is
a primitive matrix so $M_{n,N}=M^{N-n}$ is asymptotically rank 1; the dynamical
system is uniquely ergodic.  The Perron-Frobenius eigenvalue is
$\phi^2$, where $\phi$ is the golden mean; this number represents the
asymptotic volume expansion of the supertiles from level to level.

To compute the ergodic measure, it suffices to compute the 
frequencies of the $n$-supertiles and then use equation
(\ref{measure.from.frequency}) of Theorem \ref{measures.equiv.rns}
to get the frequencies of arbitrary
patches.  The vectors $\rho_n$ are the volume-normalized directions of the asymptotic
columns of $M_{n,N} =M^{N-n}$, and thus they are collinear with the right Perron-Frobenius
eigenvector of $M$, which is $\left ( \begin{smallmatrix} \phi^2 \cr  \phi \cr  \phi \cr  1 \end{smallmatrix} \right )$.
Since we have chosen unit square
prototiles, the volumes of the $n$-supertiles are $f_{n+2}^2, (f_{n+2}f_{n+1}), (f_{n+2}f_{n+1}),$ and $ 
f_{n+1}^2$ respectively.   We compute
$\rho_n = \phi^{-(2n+4)} \left ( \begin{smallmatrix}
\phi^2 \cr  \phi \cr  \phi \cr  1 \end{smallmatrix} \right )$.

Next we turn to computing the topological and measure-theoretic
spectrum.  Technically we should take the induced fusion rule that
composes two levels at once to get strong primitivity, then go up
two more levels at a time to find all of the return vectors in
$\vvv^n$.  Fortunately, in this example there
are always return vectors of the form $(f_n, 0)$ and $(0, f_n)$ (see
Figure \ref{fib.returns}).
\begin{figure}[ht]
\includegraphics[width=3.5cm]{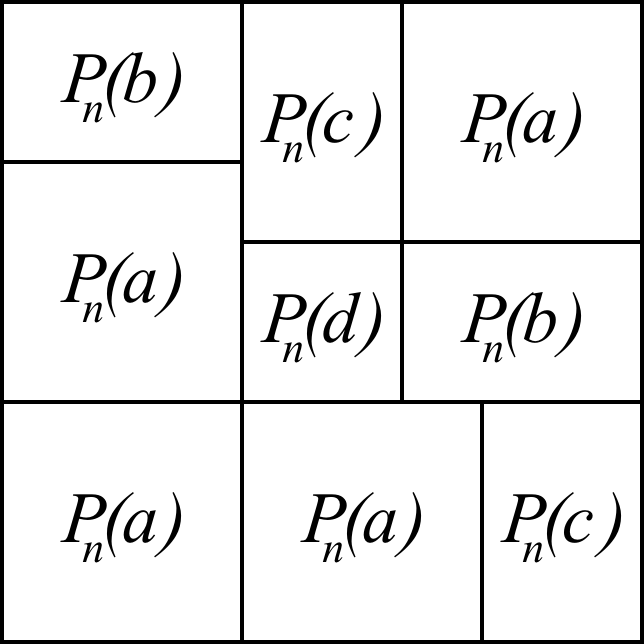}
\caption{The induced fusion for $P_{n+2}(a)$.}
\label{fib.returns}
\end{figure}
Any eigenvalue ${\vec\alpha} = (\alpha_1,\alpha_2)$, topological or
measure-theoretic, must have the property that $\lim_{n \to \infty}
{\vec\alpha} \cdot \vecv_n = 0 \pmod 1$, and this means that
$\lim_{n \to \infty} \alpha_k f_n = 0 \pmod 1$ for $k =
1,2$.  Pisot's theorem then implies that ${\vec\alpha} \in \Z[\phi] \times
\Z[\phi]$\footnote{The absence of the $\sqrt{5}$ that is present in
  Example \ref{scrambled.fib.1d} is due to the integer size of the
  prototiles. }.  For such $\vec \alpha$, the criterion for
topological eigenvalues in Theorem \ref{eig-criterion} is satisfied
because the convergence of $\eta_n({\vec\alpha})$ to 0 is exponential.
In short, the computations for both topological and measure-theoretic
eigenvalues are exactly the same as for the product of two Fibonacci
tilings (built from unit length prototiles), and we have that both
eigenvalue sets are $\Z[\phi]\times \Z[\phi]$. In particular, any
measurable eigenfunction can be chosen to be continuous.

Next we compute the \v Cech cohomology of $X_\rrr$. We use
Barge-Diamond collaring \cite{BD,BDHS}, picking the collaring radius $r_n$ to
grow slowly with $n$.  We stratify our tiling space into three pieces
$\Xi_0 \subset \Xi_1 \subset \Xi_2 = X_\rrr$. $\Xi_2$ is the entire
tiling space, $\Xi_1$ is the set of tilings where the origin is within
$r_n$ of the boundary of an $n$-supertile for every $n$, and $\Xi_0$
is the set of tilings where the origin is within a distance $r_n$ of
two supertile edges, and hence is near a supertile corner. The cohomology
of $\Xi_0$ is the cohomology of a CW complex with one cell for each 
possible pattern by which three or more 
high-order supertiles can meet at a vertex. There are 78 
such patterns. The relative cohomology of the pair $(\Xi_1, \Xi_0)$
is computed from a CW complex containing 52 cells that describe the ways that
two supertiles can meet along a common edge. The relative cohomology
of the pair $(\Xi_2, \Xi_1)$ is computed from the matrix $M$. The techniques
for generating these cells and computing the cohomology are similar to 
those for substitution tilings, and yield 
\begin{eqnarray*}
 \check H^0(\Xi_0)=\Z; \qquad & \check H^1(\Xi_0)=0; \qquad & \check H^2(\Xi_0)
=\Z^{42} \\
\check H^0(\Xi_1, \Xi_0) = 0; \qquad & \check H^1(\Xi_1, \Xi_0)= \Z^4; 
\qquad & \check H^2(\Xi_1, \Xi_0)=\Z^{18} \\
\check H^0(\Xi_1) = \Z; \qquad & \check H^1(\Xi_1)= \Z^4; 
\qquad & \check H^2(\Xi_1)=\Z^{60} \\
\check H^0(\Xi_2, \Xi_1) = 0; \qquad & \check H^1(\Xi_2, \Xi_1)= 0; 
\qquad & \check H^2(\Xi_2, \Xi_1)=\Z^{4} \\
\check H^0(X_\rrr) = \Z; \qquad & \check H^1(X_\rrr)= \Z^4; 
\qquad & \check H^2(X_\rrr)=\Z^{64}.
\end{eqnarray*}

The generators of $\check H^1(X_\rrr)=\Z^4$ are easily described. Pick a value
of $n \ge 4$.  Each edge of an $n$-supertile either has length $f_{n+1}$
or $f_{n+2}$. The first generator counts the horizontal edges of the first type,
the second generator counts the horizontal edges of the second type, and 
the third and fourth generators similarly count vertical edges. The boundaries
of two supertiles may overlap on intervals of size $f_n = f_{n+2}-f_{n+1}$ or 
$f_{n-1} = 2 f_{n+1} - f_{n+2}$. The first (or third) 
generator assigns the numbers $-1$
and $2$ to these partial edges, while the second (or fourth) 
assigns the numbers $1$ and $-1$, as these are the coefficients of $f_{n+1}$
and $f_{n+2}$.  Picking different values of $n$ gives different generators, but
the group they generate is the same.

Deformations of a tiling, by changing the shape and size of (possibly
collared) tiles, are parametrized up to mutual local derivability by $\check
H^1(X_\rrr, \R^d)$ \cite{Clark-Sadun2}.   For the Fibonacci
DPV, $\check H^1$ is the same as for the product of two 1-dimensional Fibonacci
tiling spaces, and the deformations are the same. Thus, {\em any}
deformation of the sizes and shapes yields a tiling space that is topologically
conjugate to a linear transformation of $\R^2$ applied to the original tiling
space. In particular, a self-similar version of the DPV, in which the
$a$, $b$, $c$, and $d$ tiles have dimensions $\phi \times \phi$, $\phi
\times 1$, $1\times \phi$ and $1\times 1$, is topologically conjugate
to a DPV where all tiles are congruent squares (of side $\sqrt{5}/\phi$).

The difference between the Fibonacci DPV and the product of two 1-dimensional
Fibonacci tilings is seen in the second \v Cech cohomology, where that of 
the DPV has rank 64 and that of the product has rank 4. The rank of the 
top cohomology is closely related to the independent appearance of patterns
in the tiling, via the following theorem:

\begin{thm}[\cite{Exact.regularity}]
If the rank of $\check H^d$ of a $d$-dimensional tiling space is $k$, then
there exist $k$ patterns $P_1, \ldots, P_k$, such that for any patch $P$
there exist rational numbers $c_1, \ldots, c_k$ and $c_P$ 
such that, for any region $R$ in any tiling $T$, 
$$ \#(P\text{ in }R) = \sum_{i=1}^k c_i \#(P_i \text{ in } R) + 
e(P,R),
$$
where the error term $e(P,R)$ is computable from the patterns on the boundary
of $R$, and is bounded by $c_P$ times the $(d-1)$-volume of the boundary
of $R$. 
\end{thm}
We call $P_1, \ldots, P_k$ {\em control patches}. 
For the product of two 1-dimensional Fibonacci tilings, we can take our
control patches to be the four basic tiles. 
For the Fibonacci DPV, however,
there are 60 additional control patches. They can be chosen from the generators
of $\check H^2(\Xi_0)$ and $\check H^2(\Xi_1, \Xi_0)$. That is, we have
9 control patches that involve supertiles meeting along horizontal edges,
9 that involve supertiles meeting along vertical edges, and 42 that involve
three or four supertiles meeting at a vertex. 
\end{ex}

\begin{ex} 
{\em A scrambled Fibonacci DPV.} We can construct a scrambled 
version of the Fibonacci DPV in much the same way as the 1-dimensional
scrambled Fibonacci tiling of Example \ref{scrambledFib}. We pick
an increasing sequence $N(n)$ and induce on this sequence to get an 
accelerated scrambled Fibonacci rule $\aaa$. 
We then introduce an exceptional supertile $\eee_n(e)$
at each odd level, whose population in terms of $(n-1)$-supertiles is the
same as $\aaa_n(d)$, but rearranged so that all of the $\eee_{n-1}(a)$ tiles
appear in the lower left corner, all the $\eee_{n-1}(b)$ appear in the lower
right, all the $\eee_{n-1}(c)$ appear in the upper left, and all the 
$\eee_{n-1}(d)$ appear in the upper right. On even levels, the $n$-supertiles
are built from the $(n-1)$-supertiles exactly as for the accelerated DPV, 
only with one $\eee_{n-1}(d)$ in each $n$-supertile replaced by an 
$\eee_{n-1}(e)$. Finally, we induce on even levels to obtain a prototile-regular
fusion $\sss$.

As before, if we choose the sequence $N(n)$ to grow sufficiently fast,
and if we give the prototiles the same shape as the asymptotic
supertiles, with the $a,b,c,d$ prototiles having dimensions
$\phi \times \phi$, $\phi
\times 1$, $1\times \phi$ and $1\times 1$, then the scrambled Fibonacci DPV
space is topologically weakly mixing but has pure point measurable spectrum,
being measurably conjugate to the unscrambled Fibonacci DPV.  
However, if we choose the prototiles to be unit squares, then every
$\alpha \in \Z \times \Z$ is manifestly a topological eigenvalue. 

This discrepancy means that the deformation theory for the scrambled
DPV is {\em not} the same as for the unscrambled DPV. Either the first
cohomologies are different, or, more likely, the cohomologies are
isomorphic but the two tilings have different ``asymptotically
negligible'' \cite{Clark-Sadun2} subspaces of $\check H^1(X, \R^2)$
that describe deformations that are topological conjugacies but that are 
not mutually locally derivable from the original. 
As for $\check H^2$, the rank must be
at least 64, since the control patches for the Fibonacci DPV are still
present in the scrambled DPV.
\end{ex}

This example suggests two directions for future work. One is to
understand deformation theory better, and in particular the role of
the asymptotically negligible classes. These are well-understood for
substitution tilings, but not for fusions. Another is to develop new
techniques for computing tiling cohomology for spaces that do not come from
substitutions. The Anderson-Putnam and Barge-Diamond complexes were 
defined in Section 5 for all tilings, but almost every
existing method for studying these complexes relies on 
an underlying substitution. 

\begin{ex}{\em A non-Pisot DPV.} \label{nonpisot} We base this DPV on
  the one-dimensional substitution $a \to abbb, b \to a$, which
  despite its apparent similarity to the Fibonacci DPV exhibits
  significantly different dynamical behavior.  The prototile set is
  the same as for the Fibonacci DPV, and again we choose the fusion to
  be both prototile- and transition-regular.  This time we choose our
  fusion rule at each stage to be given by $$\ppp_{n+1} =
  \left\{\raisebox{-2cm}{\includegraphics[width=4.2cm]{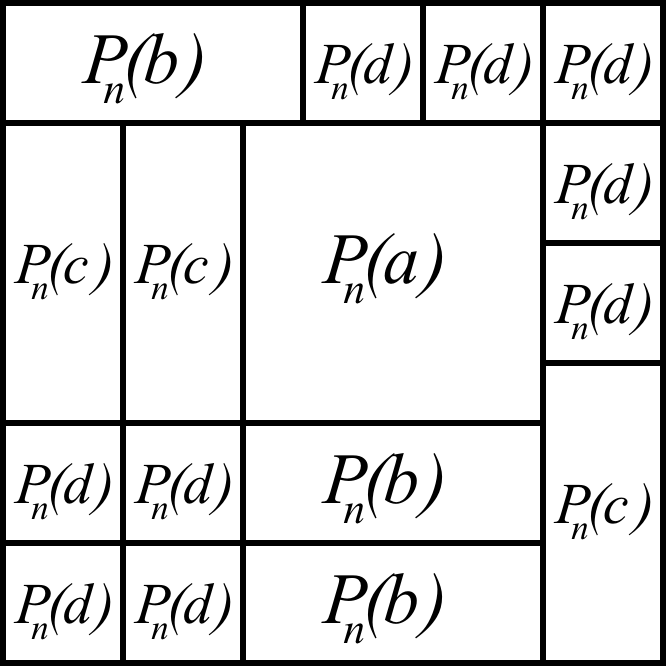}},
\raisebox{-1cm}{\includegraphics[width=4cm]{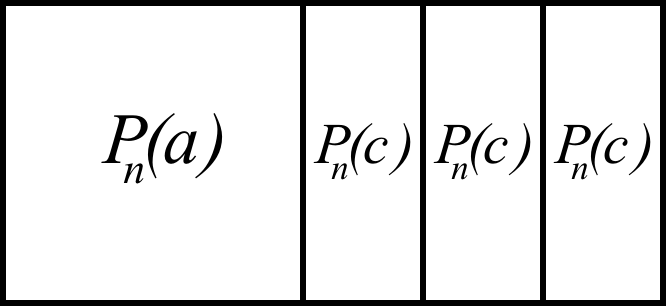}},
\raisebox{-2cm}{\includegraphics[width=1.8cm]{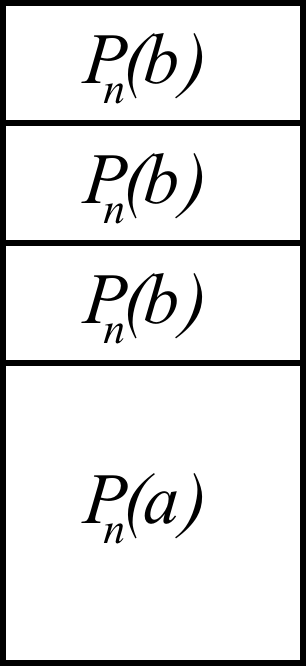}},
\raisebox{-1cm}{\includegraphics[width=1.8cm]{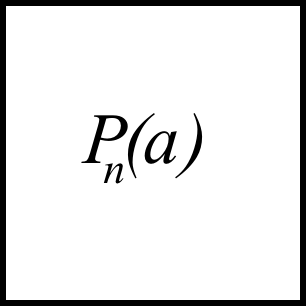}}\right\}$$
Recognizability is easily established, almost exactly as with the 
Fibonacci DPV. 

  The transition matrix has $((1 + \sqrt{13})/2)^2$ as its largest
  eigenvalue, which is not a Pisot number.  The side lengths of the
  supertiles grow as solutions to the recursion $l_{n+1} = l_n + 3s_n,
  s_{n+1} = l_n$, which are nontrivial
  linear combinations of $((1 \pm \sqrt{13})/2)^n$.  Since both of
  those numbers are greater than one in modulus, the side lengths are
  not well-approximated by powers of the positive eigenvalue.  The
  effect this has on the system is profound.  It means that the
  combinatorics of the fusion tilings are exceptionally complicated,
  in that the number of ways that $n$-supertiles can be adjacent to
  one another grows without bound as $n \to \infty$.  

This increasing complexity with scale shows up in the topology of $X_\rrr$.
Both $\check H^1$ and $\check H^2$ are infinitely generated, the
first indicating that there are infinitely many ``interesting'' deformations
of size and shape, and the second indicating that there are infinitely 
many control patterns.
Unlike the Fibonacci DPV, changes in tile sizes, while preserving
the fusion rule, can change the dynamics and in fact the topology of the 
tiling space. 
If we were to choose irrationally related side lengths for our
  prototiles, then the resulting tiling would not have finite local
  complexity \cite{Me.Robbie}.
\end{ex}

Examples \ref{FibDPV} and \ref{nonpisot} lead us to a discussion of 
the combinatorial
and geometric behavior of supertiles as $n \to \infty$.
In some cases one or the other will approach a limit as $n\to\infty$.  
Consider a prototile-regular fusion rule, and
suppose that there
is some invertible linear map $L: \R^d \to \R^d$ such that $\lim_{n
  \to \infty} L^{-n} P_n(p)$ exists for each prototile type.  If in
addition the combinatorics of how the $(n-1)$-supertiles lie
inside their $n$-supertiles stabilizes for large values of $n$,
then we call the fusion rule {\em asymptotically self-affine} (or
-similar if $L$ is a similarity).   This means that there is a self-affine
tiling that is related to the fusion tiling.   The precise nature of the
relationship varies, and no general theorems about it are known to the
authors at this time.    Both of the previous examples are asymptotically
self-similar, with the limiting prototile sets having edge lengths in the
ratios $\phi:1$ in the Fibonacci case and $(1+ \sqrt{13}):2$ in the non-Pisot
case.

A fusion tiling may have finite local complexity in the usual sense
while failing to be locally finite in an asymptotic sense. We call a
fusion rule {\em asymptotically FLC} if there is a constant $B$
such that each pair of $n$-supertiles can form at most $B$ connected
two-supertile patches. Example \ref{FibDPV} is asymptotically FLC, but
Example \ref{nonpisot} is not.  If an asymptotically self-affine tiling
is not asymptotically FLC, then the
self-affine
tiling obtained from the limiting shapes will have infinite local complexity.

\end{document}